\documentclass[12pt,centertags,oneside]{amsart}
\usepackage{amsmath,amstext,amsthm,amscd,typearea,hyperref}
\usepackage{amssymb}
\usepackage{a4wide}
\usepackage[mathscr]{eucal}
\usepackage{mathrsfs}
\usepackage{typearea}
\usepackage{charter}
\usepackage{pdfsync}
\usepackage[a4paper,width=16.5cm,top=3cm,bottom=3cm]{geometry}

\numberwithin{equation}{section}

\allowdisplaybreaks
\emergencystretch=\maxdimen
\usepackage{multicol}

\usepackage{xcolor}

\newtheorem{theorem}{Theorem}[section]
\newtheorem{definition}[theorem]{Definition}
\newtheorem{proposition}[theorem]{Proposition}
\newtheorem{corollary}[theorem]{Corollary}
\newtheorem{lemma}[theorem]{Lemma}

\newtheorem*{definition*}{Definition}

\newtheorem{mainthm}{Theorem}

\newcommand{\cali}[1]{\mathscr{#1}}

\newcommand{\GL}{{\rm GL}}

\newcommand{\supp}{{\rm supp}}

\newcommand{\diff}{{\rm d}}

\renewcommand{\Re}{\mathop{\mathrm{Re}}}

\renewcommand{\GL}{{\rm GL}}

\newcommand{\ep}{\epsilon}

\newcommand{\Leb}{{\rm Leb}}
\newcommand{\Lip}{\mathop{\mathrm{Lip}}\nolimits}

\newcommand{\Ac}{\cali{A}}
\newcommand{\Bc}{\cali{B}}
\newcommand{\Cc}{\cali{C}}

\newcommand{\Ec}{\cali{E}}

\newcommand{\Gc}{\cali{G}}
\newcommand{\Hc}{\cali{H}}
\newcommand{\Ic}{\cali{I}}
\newcommand{\Jc}{\cali{J}}

\newcommand{\Lc}{\cali{L}}

\newcommand{\Tc}{\cali{T}}

\newcommand{\B}{\mathbb{B}}

\newcommand{\C}{\mathbb{C}}

\newcommand{\N}{\mathbb{N}}
\newcommand{\Z}{\mathbb{Z}}
\newcommand{\Q}{\mathbb{Q}}
\newcommand{\R}{\mathbb{R}}

\renewcommand\P{\mathbb{P}}

\newcommand{\E}{\mathbf{E}}

\newcommand{\lp}{\langle}
\newcommand{\rp}{\rangle}

\newcommand{\norm}[1]{\lVert#1\rVert}

\newcommand{\oA}{\mathcal{A}}

\newcommand{\oP}{\mathcal{P}}
\newcommand{\oN}{\mathcal{N}}
\newcommand{\oR}{\mathcal{R}}
\newcommand{\oE}{\mathcal{E}}
\newcommand{\oU}{\mathcal{U}}
\newcommand{\oQ}{\mathcal{Q}}
\newcommand{\oS}{\mathcal{S}}
\newcommand{\oL}{\mathcal{L}}

\title[Berry-Esseen bounds with targets and LLT for products of random matrices]{Berry-Esseen bounds with targets and Local Limit Theorems for products of random matrices}

\author{Tien-Cuong Dinh}
\address{Department of Mathematics,  National University of Singapore - 10, Lower Kent Ridge Road - Singapore 119076}
\email{matdtc@nus.edu.sg}

\author{Lucas Kaufmann}
\address{Center for Complex Geometry - Institute for Basic Science (IBS) - 55 Expo-ro Yuseong-gu Daejeon 34126 South Korea}
\email{lucaskaufmann@ibs.re.kr}

\author{Hao Wu}
\address{Department of Mathematics,  National University of Singapore - 10, Lower Kent Ridge Road - Singapore 119076}
\email{matwu@nus.edu.sg}

\thanks{This work was supported by the NUS and MOE grants  AcRF Tier 1 R-146-000-319-114 and MOE-T2EP20120-0010. L. Kaufmann was supported by the Institute for Basic Science (IBS-R032-D1)}

\begin{document}

\begin{abstract}
Let $\mu$ be a probability measure on $\GL_d(\R)$ and denote by $S_n:= g_n \cdots g_1$ the associated random matrix product, where $g_j$'s are i.i.d.'s with law $\mu$. We study  statistical properties of random variables of the form $$\sigma(S_n,x) + u(S_n x),$$ where $x \in \P^{d-1}$, $\sigma$ is the norm cocycle and $u$ belongs to a class of admissible functions on $\P^{d-1}$ with values in $\R \cup \{\pm \infty\}$.  

 Assuming that $\mu$ has a finite exponential moment and generates a proximal and strongly irreducible semigroup, we obtain optimal Berry-Esseen bounds and the Local Limit Theorem for such variables using a large class of observables on $\R$ and H\"older continuous target functions on $\P^{d-1}$. As particular cases, we obtain new limit theorems for $\sigma(S_n,x)$ and for the coefficients of $S_n$.
\end{abstract}

\clearpage\maketitle
\thispagestyle{empty}

%\setcounter{tocdepth}{1}
%\tableofcontents

\section{Introduction}

Let $\mu$ be a probability measure on $G:=\GL_d(\R)$ with $d \geq 2$.  Consider the random walk on $G$ induced by $\mu$ given by
$$S_n: = g_n \cdots g_1,$$ 
where  $n \geq 1$ and the $g_j$'s are independent and identically distributed (i.i.d.)\ random elements of $G$ with law  $\mu$.  
The study of the statistical properties of various natural functions associated with $S_n$ is a well-developed topic with an intense recent activity.  We refer to \cite{bougerol-lacroix,benoist-quint:book} for  an overview  of the fundamental results and the reference list for some of the recent results.

The group $G$ acts naturally on the real projective space $\P^{d-1}$ and, in order to study the random walk induced by $\mu$, one usually considers various real valued functions on $G$ and $\P^{d-1}$. An important function in this setting is the \textit{norm cocycle} $\sigma: G \times \P^{d-1} \to \R$, given by
\begin{equation} \label{eq:cocycle-def}
\sigma(g,x) = \sigma_g(x):= \log \frac{\norm{gv}}{\norm{v}}, \quad \text{for }\,\, v \in \R^d \setminus \{0\}, \, x = [v] \in \P^{d-1}   \, \text{ and } g \in G.
\end{equation}

The cocycle relation $\sigma(g_2g_1,x) = \sigma(g_2,g_1 \cdot x) + \sigma(g_1,x)$ can be used  to effectively apply methods such as the spectral theory of complex transfer operators (see Subsection \ref{subsec:markov-op}) and martingale approximation \cite{benoist-quint:CLT}. This yields several limit theorems for the sequence of random variables $\sigma(S_n,x)$ analogous to classical limit theorems for sums of real i.i.d's. From that, similar limit theorems can be obtained for other natural random variables, such as $\log \| S_n \| = \sup_{x \in \P^{d-1}} \sigma(S_n,x)$, the spectral radius of $S_n$, the coefficients of $S_n$, etc.

The goal of this work is to obtain limit theorems for random variables of the form 
\begin{equation} \label{eq:sigma+u}
\sigma(S_n,x) + u(S_n x),
\end{equation}
where $u$ belongs to a class of admissible functions on $\P^{d-1}$ with values in $\R \cup \{\pm \infty\}$ (see Definition \ref{def:admissible}). These include the variables $\sigma(S_n,x)$, corresponding to $u=0$, and also the logarithm of the absolute value of the coefficients of $S_n$, as we now explain.

  For $v \in \R^d$ and $f \in (\R^d)^*$,  its dual space,  we denote by $\lp f,v \rp := f(v)$ their natural coupling.  Observe that the $(i,j)$-entry of the matrix $S_n$ is given by $\lp e_i^* , S_n  e_j \rp$,  where $(e_k)_{1\leq k \leq d}$ (resp. $(e^*_k)_{1\leq k \leq d}$) denotes the canonical basis of $\R^d$ (resp.$(\R^d)^*$).   More generally, one can study random variables of the form
$$\log{ |\lp f, S_n v\rp | \over \norm{f} \norm{v}},$$
with  $v \in \R^d \setminus \{0\}$ and $f \in (\R^d)^* \setminus \{0\}$. We refer to \cite{benoist-quint:book,cuny-dedecker-merlevede-peligrad-2,DKW:LLT,DKW:BE-LLT-coeff,grama-quint-xiao,xiao-grama-liu:coeff} for limit theorems for the above variables. One can easily check that
 \begin{equation} \label{eq:coeff-split}
\log{ |\lp f, S_n v \rp | \over \norm{f} \norm{v}} = \sigma(S_n,x) + \log \delta(S_n x,y),
\end{equation}
where $x = [v] \in \P^{d-1}$, $y = [f] \in (\P^{d-1})^*$ and $\delta(x,y):= \frac{ |\lp f, v \rp |}{\norm{f} \norm{v}}$.  Moreover, $\delta(x,y) = d (x,H_y)$, where $H_y := \P(\ker f)$ is the hyperplane in $\P^{d-1}$ defined by $y$ and $d$ is a natural distance on $\P^{d-1}$ (see Section \ref{sec:prelim}). In particular, we obtain random variables of the form \eqref{eq:sigma+u} by choosing $u(x) = \log d(x,H_y)$.

In order to obtain meaningful results,  some standard assumptions on the measure $\mu$ need to be made.   Recall that a matrix $g\in G$ is said to be \textit{proximal} if it admits a unique eigenvalue of maximal modulus which is moreover of multiplicity one.  Let  $\Gamma_\mu$ be the smallest closed semigroup containing the support of $\mu$.  We assume that $\Gamma_\mu$ is \textit{proximal}, that is,  it contains a proximal matrix,  and  \textit{strongly irreducible},  that is,  the action of $\Gamma_\mu$ on $\R^d$ does not preserve a finite union of proper linear subspaces.  It is well-known that, under the above conditions, $\mu$ admits a unique stationary probability measure on $\P^{d-1}$, denoted by $\nu$, see Section \ref{sec:prelim}.  We'll also assume that $\mu$ has a \textit{finite exponential moment}, that is,  $\int_{G} N(g)^\varepsilon \diff\mu(g) < \infty$ for some $\varepsilon>0$, where $N(g):=\max\big( \norm{g},\norm{g^{-1}} \big)$. 
 
Recall that the \textit{first Lyapunov exponent} associated with $\mu$ is the following limit,  which is known to exist
$$\gamma := \lim_{n \to \infty} \frac1n \,\E \big( \log\|S_n\| \big)=\lim_{n \to \infty} \frac1n \int \log\|g_n \cdots g_1\| \, \diff \mu(g_1) \cdots \diff \mu(g_n).$$
Furstenberg-Kesten's theorem \cite{furstenberg-kesten} says that the sequence $\frac1n \log\|S_n\|$ converges to $\gamma$ almost surely. This is the analogue of the Law of Large Numbers for sums of real i.i.d.'s.  Under our assumptions, the norm cocycle also satisfies the Central Limit Theorem (CLT), as shown by Le Page in \cite{lepage:theoremes-limites},  see also \cite{benoist-quint:CLT} for a new proof which holds for measures  with a finite second moment.  The theorem says that,  for every $x\in \P^{d-1}$,  we have that
\begin{equation} \label{eq:CLT-intro}
{1\over \sqrt{n}} \big(\sigma(S_n,x)-n\gamma\big) \longrightarrow \cali N(0;\varrho^2)
\end{equation}
in law as $n \to \infty$, where $\varrho>0$ is a constant and $\cali N(0;\varrho^2)$ denotes the centered normal distribution with variance $\varrho^2$.  Le Page also obtained a ``coupled"  CLT  for the random variable $\big(\sigma(S_n,x), S_n x\big) \in \R \times \P^{d-1}$, in the spirit of Theorem \ref{thm:BE-general} below.  A CLT for the coefficients of $S_n$ can also be obtained,  see \cite{guivarch-raugi,benoist-quint:book,DKW:BE-LLT-coeff,xiao-grama-liu:coeff}.

Once the CLT has been established,  it is natural to ask whether we can estimate the speed of convergence in \eqref{eq:CLT-intro}. In the present work, we obtain such bounds for various test functions on the couple $\big(\sigma(S_n,x) + u(S_n), S_n x\big)$,  as announced in \cite{DKW:BE-LLT-coeff}.  In analogy with the case of sums of real i.i.d.'s,  the best bound one can hope for is $O(1 / \sqrt n)$.  For the norm cocycle,  some cases were obtained by Le Page  \cite{lepage:theoremes-limites} and Xiao-Grama-Liu \cite{xiao-grama-liu},  see below for more details.   For the coefficients,  this question was treated by the authors in \cite{DKW:BE-LLT-coeff}.  Later, a version for the couple $\big(\log{ |\lp f, S_n v \rp | \over \norm{f} \norm{v}}, S_n x \big)$ was announced in \cite{xiao-grama-liu:coeff}.  We deal here with more general test functions which include these  known cases.  See also \cite{cuny-dedecker-merlevede-peligrad,cuny-dedecker-merlevede-peligrad-2,DKW:LLT} for partial results under lower moment conditions.

For general admissible functions $u$, the above results are delicate to obtain, as the case  $u(x) = \log d(x,H_y)$ already reveals. In particular, they cannot be easily deduced form the corresponding theorems for $\sigma(S_n,x)$. The main challenge comes from the appearance of singularities in the phase space $\P^{d-1}$, see e.g.\ \cite{grama-quint-xiao}.  In the recent work \cite{DKW:BE-LLT-coeff}, we have introduced new ideas allowing us to bypass this difficulty, yielding sharp bounds.  As mentioned in that paper,  our method can be pushed further to produce similar Berry-Esseen bounds for  general admissible functions $u$ and large class of test functions on $\R \times \P^{d-1}$. This is the content of the present work.

We now state our main results. We refer to Section \ref{sec:admissible-fcns} for the definition of the class of admissible functions $\Lc(\eta_*,\alpha_*,A_*)$. In particular, the functions $u=0$ and $u(x) = \log d(x,H_y)$, where $H_y \subset \P^{d-1}$ is a hyperplane as above, are admissible for suitable choices of $\eta_*,\alpha_*,A_*$, see Lemma \ref{lemma:log-dist-admissible}. Therefore, our results hold for the norm cocycle and the coefficients of $S_n$. Moreover, in contrast with previous works, our proofs show that these cases can be treated simultaneously.

Our first result is a Berry-Esseen bound with optimal rate $O(1/\sqrt n)$ for the sequence of $\R \times \P^{d-1}$-valued random variables $\big(\sigma(S_n,x) + u(S_n x), S_n x \big)$. This is phrased in terms of general test functions, or ``target functions'', on $\R \times \P^{d-1}$.    Denote by $\Hc$ the space consisting of bounded Lipschitz functions $\psi$ on $\R$ such that  $\psi'$ is bounded and belongs to $L^1(\R)$.  Define a norm on $\Hc$ by 
$$\norm{\psi}_\Hc:=\norm{\psi}_\infty+\norm{\psi'}_\infty +\norm{\psi'}_{L^1}.$$ 
For an interval $J\subset\R$,  set  $\psi_J(t):= \mathbf 1_{t\in J}\cdot \psi(t)$. For $0<\alpha \leq 1$, let  $\Cc^{\alpha}(\P^{d-1})$ be the space of $\alpha$-H\"older continuous functions on $\P^{d-1}$.
 
\begin{mainthm}[Berry-Esseen bound with targets]\label{thm:BE-general}
	Let $\mu$ be a  probability measure on $\GL_d(\R)$.  Assume that $\mu$ has a finite exponential moment and  that $\Gamma_\mu$ is proximal and strongly irreducible.  Let $\gamma, \varrho$ and $\nu$ be as above.  Let $\eta_*>0$, $0<\alpha_*\leq 1$ and $A_*>0$ be constants. 
	Then, there are constants $0<\alpha \leq \alpha_*$ and $C>0$ such that, for any $\psi\in \Hc$, $\varphi\in\Cc^{\alpha}(\P^{d-1})$, $u\in \Lc(\eta_*,\alpha_*,A_*)$, any interval $J\subset\R$, and all $n\geq 1$, we have
\begin{align*}
	\bigg| \mathbf E \Big(  \psi_J\Big( {\sigma(S_n,x)+u(S_nx) - n \gamma \over \sqrt n}\Big)\cdot\varphi( S_n x  )\Big)  -  \frac{1}{\sqrt{2 \pi} \, \varrho} \int_{\P^{d-1}}\varphi \,\diff\nu \int_{J} e^{-\frac{s^2}{2 \varrho^2}}\psi(s)\, \diff s \bigg|  \\ \leq \frac{C}{\sqrt n}\norm{\psi}_\Hc\norm{\varphi}_{\Cc^\alpha}.
\end{align*}	
\end{mainthm}

The above theorem contains simultaneously the Berry-Esseen bound with targets for the norm cocycle and the coefficients of $S_n$. It generalizes previous results when $\psi = \mathbf 1$ and $u=0$ or $u(x) = \log d(x,H_y)$. For $u=0$,  this is due to Le Page when $\psi=\mathbf 1$ and $\varphi = \mathbf 1$,  see \cite{lepage:theoremes-limites,bougerol-lacroix}  and Xiao-Grama-Liu \cite{xiao-grama-liu} when $\psi=\mathbf 1$  and $\varphi$ is a general H\"older function.  For $u(x) = \log d(x,H_y)$, which encodes the coefficients of $S_n$,  the case $\psi=\mathbf 1$ and $\varphi = \mathbf 1$ was obtained very recently by the authors in \cite{DKW:BE-LLT-coeff} and, as mentioned in that paper, the method developed there turns out to be effective in the more general setting detailed here.  Using similar arguments,  Xiao-Grama-Liu have recently obtained the same result when  $\psi=\mathbf 1$  and  $\varphi$ is general  \cite{xiao-grama-liu:coeff}.

The difficulty in proving Berry-Esseen bounds in the presence of non-constant target functions $\varphi$ is explained in \cite[Subsection 1.2]{xiao-grama-liu}.  Roughly speaking,  when applying the Berry-Esseen lemma \cite[XVI.3]{feller:book} directly,  the integrand displays  a singularity at the origin that cannot be easily handled. This is already observed in the simplest case when $u=0$. To overcome this,  Xiao-Grama-Liu \cite{xiao-grama-liu} consider a complex contour around the origin and use Cauchy integral formula together with  the so called saddle point method.   In the present  work,  inspired by \cite{DKW:BE-LLT-coeff},  we observe that the above mentioned singularity is actually ``removable" by considering instead the Cauchy principal value of the integral along the imaginary line, see  Lemma \ref{lemma:BE-feller} below.  This provides a simple solution to the above technical difficulty.  In particular,  we are not required to consider perturbations of the Markov operator along the real axis.

It is not difficult to deduce from the above result weaker Berry-Esseen bounds with rate $O(\log n/\sqrt n)$ for the random variable  $\log\norm{S_n}$ and for the spectral radius of $S_n$ and similar test functions.  Since this rate is likely not optimal,  we choose not to consider these questions in this article.   Using the techniques developed here and the results of \cite{DKW:LLT},  we can also obtain analogues of Theorem \ref{thm:BE-general}  for random matrices in $\GL_2(\R)$ or $\GL_2(\C)$ under a {\it finite third moment condition},  which is the minimal moment condition one should require for this problem. 

\medskip

Theorem \ref{thm:BE-general} implies immediately the Central Limit Theorem for $\sigma(S_n,x) + u(S_n x)$, namely
\begin{equation*}
{1\over \sqrt{n}} \big(\sigma(S_n,x) + u(S_n x) -n\gamma\big) \longrightarrow \cali N(0;\varrho^2)
\end{equation*}
in law as $n \to \infty$. The corresponding local version of this theorem is given by the following result.

\begin{mainthm}[Local Limit Theorem] \label{thm:LLT-general}
	Let $\mu$ be a  probability measure on $\GL_d(\R)$.  Assume that $\mu$ has a finite exponential moment and  that $\Gamma_\mu$ is proximal and strongly irreducible.  Let $\gamma, \varrho$ and $\nu$ be as above. 
	Then,  for any $u\in \Lc(\eta_*,\alpha_*,A_*)$ and any continuous function $\Phi$ with compact support on $\R\times \P^{d-1}$, we have
	$$\lim_{n\to \infty}\sup_{t\in\R}\bigg| \sqrt{2\pi n}  \,\varrho \,\mathbf E \Big( \Phi\Big( t+ \sigma(S_n,x)+u(S_nx) - n \gamma, S_n x\Big) \Big)  -   e^{-\frac{t^2}{2 \varrho^2 n}}  \int_{\R\times\P^{d-1}} \Phi(s,w)\,\diff s\diff \nu(w)\bigg| =0.$$
\end{mainthm}

As before, the choices $u=0$ and $u(x) = \log d(x,H_y)$ yield the Local Limit Theorem (LLT) for the norm cocycle and the coefficients of $S_n$, respectively. Our result covers the known cases, corresponding to particular choices of $\Phi$, see  \cite{benoist-quint:book,grama-quint-xiao,DKW:BE-LLT-coeff,xiao-grama-liu:coeff}.  

The above theorems can be seen as mixed versions of the corresponding limit theorem for $\sigma(S_n,x) + u(S_n x)$ with the  equidistribution property of orbits $S_n x$ towards  the Furstenberg measure $\nu$ on $\P^{d-1}$, see the comments after Theorem \ref{thm:spectral-gap}. 

\medskip

Our methods also give a Local Limit Theorem with moderate deviations, that we discuss in  Section \ref{sec:LLT-D}.

\medskip

\noindent\textbf{Notations.} Throughout this article, the symbols $\lesssim$ and $\gtrsim$ stand for inequalities up to a multiplicative constant.  The dependence of these constants on certain parameters (or lack thereof),  if not explicitly stated,  will be clear from the context.  We denote by $\mathbf E$ the expectation and $\mathbf P$ the probability.

\section{Preliminary results} \label{sec:prelim}

This section contains some known results and the necessary background material for the proof of the main theorems. For the details, the read may refer to \cite{benoist-quint:book,bougerol-lacroix,DKW:BE-LLT-coeff,lepage:theoremes-limites}. We always assume  that $\mu$ has a finite exponential moment and  that $\Gamma_\mu$ is proximal and strongly irreducible.

Under our assumptions, $\mu$ admits a unique \textit{stationary measure}, also known as \textit{Furstenberg measure}. This is the unique probability measure $\nu$ on $\P^{d-1}$ satisfying $$\int_G g_* \nu \, \diff \mu(g)= \nu.$$

\subsection{Large deviation estimates and regularity of the stationary measure}

The following large deviation estimates will be used later. 

\begin{proposition}[\cite{benoist-quint:book}--Theorem 12.1]  \label{prop:BQLDT}
For any $\ep>0$ there exist $c>0$ and $n_0 \in\N$ such that, for all $n\geq n_0$ and $x\in \P^{d-1}$, one has 
	$$   \mu^{*n} \big\{g\in G:\, |\sigma(g,x)-n\gamma|  \geq n\ep \big\}\leq e^{-cn}         . $$
\end{proposition}

We equip $\P^{d-1}$ with the distance
\begin{equation*} 
d(x,w) : = \sqrt{1 - \bigg( \frac{\langle v_x,v_w \rangle}{\|v_x\| \|v_w\|} \bigg)^2}, \quad \text{where} \quad v_x,v_w \in \R^d \setminus \{0\}, \, x = [v_x], \,\, w = [v_w] \in \P^{d-1}.
\end{equation*}
Note that $d(x,w)$ is the sine of the angle between the lines $x$ and $w$ in $\R^d$ and  $(\P^{d-1}, d)$ has diameter one.

For a hyperplane $H$ in $\P^{d-1}$ and $r>0$,  we denote $\B(H,r) :=\big\{x \in \P^{d-1}: d(x,H) < r\big\}$, which is a ``tubular'' neighborhood  of $H$. Also, we denote by $(\P^{d-1})^*$ the projectivization of $(\R^d)^*$ and, for $y\in (\P^{d-1})^*$, we denote by $H_y$ the kernel of $y$, which is a (projective) hyperplane in $\P^{d-1}$.
The following is an important regularity property of the  stationary measure $\nu$.

\begin{proposition}[\cite{guivarch:1990}, \cite{benoist-quint:book}--Theorem 14.1]\label{prop:regularity}
There are constants $C>0$ and $\eta>0$ such that 
	$$\nu\big(\B(H_y,r)\big)\leq C r^\eta \quad\text{for every} \quad y\in (\P^{d-1})^* \, \, \text{ and } \,\, 0 \leq r \leq 1.$$
\end{proposition}

\subsection{The Markov operator and its perturbations} \label{subsec:markov-op}

The \textit{Markov operator} associated to $\mu$ is the operator 
$$\oP \varphi(x):=\int_{G} \varphi(gx) \,\diff\mu(g),$$
acting on functions on $\P^{d-1}$.
For $z\in\C$, we also consider the perturbation $\oP_z$ of $\oP$ (also called \textit{complex transfer operator}) given by
\begin{equation} \label{eq:markov-op-def}
\oP_z \varphi(x):=\int_{G} e^{z\sigma(g,x)}\varphi(gx) \,\diff\mu(g),
\end{equation}
where $\sigma(g,x)$ is the norm cocycle defined in \eqref{eq:cocycle-def}. Notice that $\oP_0= \oP$ and a direct computation using the cocycle relation $\sigma(g_2g_1,x) = \sigma(g_2,g_1  x) + \sigma(g_1,x)$ gives that
\begin{equation} \label{eq:markov-op-iterate}
\oP^n_z \varphi (x)   = \int_G e^{z \sigma(g,x)} \varphi(gx) \, \diff \mu^{* n} (g),
\end{equation}
where $\mu^{* n}$ is the convolution power of $\mu$, obtained by projecting the product measure $\mu^{\otimes n}$ on $G^n$ to $G$ via the map $(g_n,\ldots,g_1) \mapsto g_n \cdots g_1$.
\medskip

We now state some fundamental results of Le Page about the spectral properties of the above operators acting on some H\" older spaces. For $0<\alpha<1$, we denote by $\Cc^\alpha(\P^{d-1})$ the space of H\"older continuous functions on $\P^{d-1}$ equipped with the norm
\begin{equation*}
\|\varphi\|_{\Cc^\alpha} := \|\varphi\|_\infty + \sup_{x \neq y \in \P^{d-1}} \frac{|\varphi(x)-\varphi(y)|}{d(x,y)^\alpha}. 
\end{equation*}

 Recall that the essential spectrum of an operator is the subset of the spectrum obtained by removing its isolated points corresponding to eigenvalues of finite multiplicity.
The essential spectral radius $\rho_{\rm ess}$ is the radius of the smallest disc centred at the origin which contains the essential spectrum.

\begin{theorem}[\cite{lepage:theoremes-limites}, \cite{bougerol-lacroix}-V.2] \label{thm:spectral-gap} There exists an $0<\alpha_0 \leq 1$ such that, for all $0<\alpha \leq \alpha_0$, the operator $\oP$ acts continuously on $\Cc^\alpha(\P^{d-1})$ with a spectral gap. More precisely, the essential spectral radius of $\oP$ satisfies $\rho_{\rm ess}(\oP)<1$ and $\oP$ has a single eigenvalue of modulus $\geq 1$ located at $1$, which is isolated and of multiplicity one.
\end{theorem}

It follows directly from the above theorem that $\|\oP^n - \oN\|_{\Cc^\alpha} \leq C \tau^n$ for some constants $C > 0$ and $0<\tau<1$, where $\oN$ is the projection onto the space of constant functions defined by $\varphi \mapsto \big( \int_{\P^{d-1}} \varphi \,  \diff \nu \big) \cdot \mathbf 1$. 
The following result says that the spectral properties of $\oP_0$ persist for $\oP_z$ for $z$ near the origin and gives a useful decomposition of $\oP_z$,  see e.g.\ \cite[V.4]{bougerol-lacroix}.

\begin{proposition} \label{prop:spectral-decomp}
 There exist $0<\alpha_0 \leq 1$ and $b > 0$ such that for $|\Re z| < b$, the operators $\oP_z$ act continuously on $\Cc^\alpha(\P^{d-1})$ for all $0<\alpha \leq \alpha_0$. Moreover, the family of operators $z \mapsto \oP_z$ is analytic near $z=0$.
	
	In particular, there exists an $\epsilon_0 > 0$ such that, for $|z|\leq \epsilon_0$, one has a decomposition
	\begin{equation*} 
	\oP_z = \lambda_z \oN_z + \oQ_z,
	\end{equation*}
	where $\lambda_z \in \C$, $\oN_z$ and $\oQ_z$ are bounded operators on $\Cc^{\alpha}(\P^{d-1})$ and 
	\begin{enumerate}
		\item $\lambda_0 = 1$ and $\oN_0 \varphi = \int_{\P^{d-1}} \varphi \, \diff \nu$, which is a constant function,  where $\nu$ is the unique $\mu$-stationary measure;
		\item $\rho:= \displaystyle \lim_{n \to \infty
		} \|\oP_0^n - \oN_0\|_{\Cc^\alpha}^{1 \slash n} < 1$;
		
		\item $\lambda_z$ is the unique eigenvalue of maximum modulus of $\oP_z$, $\oN_z$ is a rank-one projection and $\oN_z \oQ_z = \oQ_z \oN_z = 0$;
		
		\item the maps $z \mapsto \lambda_z$,  $z \mapsto \oN_z$ and $z \mapsto \oQ_z$ are analytic;
		
		\item  $|\lambda_z| \geq \frac{2 + \rho}{3}$ and for every $k\in\N$, there exists a constant $c > 0$ such that $$\Big \| \frac{\diff^k \oQ_z^n}{\diff z^k} \Big \|_{\Cc^\alpha} \leq c \Big( \frac{1 + 2 \rho}{3} \Big)^n \quad \text{ for every}\quad n \geq 0;$$
		
		\item for $z=i\xi\in i\R$, we have
		$$   \lambda_{i\xi} = 1 + i  \gamma \xi - \frac{\varrho^2+\gamma^2}{2}\xi^2+O(|\xi|^3) \quad \text {as } \,\, \xi \to 0,$$ where $\gamma$ is the first Lyapunov exponent of $\mu$ and $\varrho^2 > 0$ is a constant.
	\end{enumerate}
\end{proposition}

The above constant  $\varrho^2 > 0$ coincides with the variance in the Central Limit Theorem for the norm cocycle \eqref{eq:CLT-intro}, see \cite{bougerol-lacroix,benoist-quint:book,DKW:LLT}. As a consequence, we have the following estimates, see \cite[Lemma 4.9]{DKW:LLT} and \cite[Lemma 9]{lepage:theoremes-limites}.
\begin{lemma}\label{lemma:lambda-estimates}
	Let $\ep_0$ be as in Proposition \ref{prop:spectral-decomp}. There exists $0 < \xi_0 < \ep_0$ such that, for all $n \in \N$ large enough, one has $$\big|\lambda_{{i\xi\over \sqrt n}}^n\big|\leq e^{-{\varrho^2\xi^2\over 3}} \quad\text{for}\quad  |\xi|\leq \xi_0\sqrt n,$$
	$$\Big|  e^{-i\xi\sqrt n \gamma}\lambda_{{i\xi\over \sqrt n}}^n-e ^{-{\varrho^2\xi^2\over 2}}  \Big|\leq {c\over \sqrt n}|\xi|^3e^{-{\varrho^2\xi^2\over 2}} \quad\text{for}\quad |\xi|\leq \sqrt[6] n,$$  
	$$\Big| e^{-i\xi\sqrt n \gamma} \lambda_{{i\xi\over \sqrt n}}^n-e ^{-{\varrho^2\xi^2\over 2}}  \Big|\leq {c\over \sqrt n}e^{-{\varrho^2\xi^2\over 4}} \quad\text{for}\quad \sqrt[6] n<|\xi|\leq \xi_0\sqrt n,$$  
	where  $c>0$ is a constant independent of $n$.
\end{lemma}

The following important result describes the spectrum of $\oP_{i\xi}$ for large real values of $\xi$. It will be used in the proof of Theorem \ref{thm:LLT-general}.

\begin{proposition}[\cite{lepage:theoremes-limites}, \cite{benoist-quint:book}-Chapter 15] \label{prop:spec-Pxi}
Let $\mu$ be as in Theorem \ref{thm:spectral-gap} and $0<\alpha_0 \leq 1$ be a small constant. Let $K$ be a compact subset of $\R \setminus \{0\}$. Then, for every $0<\alpha \leq \alpha_0$ there exist constants $C_K>0$ and $0<\rho_K<1$ such that $\norm{\oP^n_{i\xi}}_{\Cc^\alpha}\leq C_K \rho_K^n$ for all $n\geq 1$ and $\xi\in K$.
\end{proposition}

\subsection{Fourier transform and characteristic functions}

Recall that the Fourier transform of an integrable function $f$, denoted by $\widehat f$, is defined by
$$\widehat f(\xi):=\int_{-\infty}^{\infty}f(t)e^{-i t\xi} \diff t$$ 
and the inverse Fourier transform is 
$$f(t)={1\over {2\pi}}\int_{-\infty}^{\infty} \widehat f (\xi) e^{ i t\xi} \diff \xi.$$
Note that the Fourier transform of $\widehat f(\xi)$ is $2\pi f(-t)$ and the Fourier transform exchanges the convolution and pointwise product: $\widehat{f_1*f_2}=\widehat f_1\cdot \widehat f_2$.

\begin{lemma}[\cite{DKW:LLT}--Lemma 2.2] \label{lemma:vartheta}
	There exists a smooth strictly positive even function $\vartheta$ on $\R$ with $\int_\R \vartheta(t) \diff t=1$ such that its Fourier transform $\widehat\vartheta$ is  a smooth even function supported by $[-1,1]$.
	Moreover, for $0<\delta \leq 1$ and $\vartheta_\delta(t):=\delta^{-2}\vartheta(t/\delta^2)$, we have that $\widehat{\vartheta_\delta}$ is supported by $[-\delta^{-2},\delta^{-2}]$, $|\widehat{\vartheta_\delta}|\leq 1, \norm{\widehat{\vartheta_\delta}}_{\Cc^1}\leq c$ and $\int_{|t|\geq \delta} \vartheta_\delta (t)\,\diff t\leq c\delta^2$ for some constant $c>0$ independent of $\delta$.
\end{lemma}

We will need the following elementary lemma.

\begin{lemma}\label{lemma:convolution-lip}
Let $0<\delta \leq 1$ and $\psi$ be a Lipschitz function on $\R$ with Lipschitz norm bounded by $1$. Then there exists a constant $C>0$ independent of $\psi$ and $\delta$ such that 
$$ \big|\psi* \vartheta_\delta (t)- \psi(t)\big|\leq C \delta^2  \quad\text{for all}\quad t\in\R.       $$
\end{lemma}
\proof
By definition and  Lemma \ref{lemma:vartheta}, $\big|\psi* \vartheta_\delta (t)- \psi(t)\big|$ is equal to
\begin{align*}
&\Big| \int \psi(t-s) \vartheta_\delta (s) \, \diff s-\int \psi(t) \vartheta_\delta(s)  \,\diff s \Big| =\Big| \int \big( \psi(t-s)-\psi(t)  \big) \vartheta_\delta (s) \,\diff s  \Big| \\
&\leq \Big| \int_{|s|\leq \delta} \big( \psi(t-s)-\psi(t)  \big) \vartheta_\delta (s) \,\diff s  \Big|+\Big| \int_{|s|\geq \delta} \big( \psi(t-s)-\psi(t)  \big) \vartheta_\delta (s) \,\diff s  \Big|\\
&\leq \int_{|s|\leq \delta} |s|\vartheta_\delta(s)\,\diff s +  \int_{|s|\geq \delta} 2\vartheta_\delta (s) \,\diff s =\delta^2\int_{|s|\leq \delta^{-1}}|s|\vartheta(s)\,\diff s+  \int_{|s|\geq \delta} 2\vartheta_\delta (s) \,\diff s \lesssim \delta^2,
\end{align*}
where in the last line we have used the change of variables $s\mapsto \delta^{-2} s$ and the fact that $|s|\vartheta(s)$ is integrable since $\vartheta$ has fast decay at $\pm \infty$. The proof of the lemma is finished.
\endproof

The following approximation result will be needed in the proof of Theorem \ref{thm:LLT-general}.

\begin{lemma}[\cite{DKW:LLT}--Lemma 2.4] \label{lemma:conv-fourier-approx}
	Let $\psi$ be a continuous real-valued function with support in a compact set $K$ in $\R$.  Assume that $\|\psi\|_\infty \leq 1$. Then, for every  $0< \delta \leq 1$ there exist  smooth functions $\psi^\pm_\delta$  such that $\widehat {\psi^\pm_\delta}$ have support in $[-\delta^{-2},\delta^{-2}]$,  $$\psi^-_\delta \leq\psi\leq \psi^+_\delta,\quad \lim_{\delta \to 0} \psi^\pm_\delta =\psi    \quad \text{and} \quad  \lim_{\delta \to 0} \big \|\psi^\pm_\delta -\psi \big \|_{L^1} = 0.$$ 
	Moreover,  $\norm{\psi_\delta^\pm}_\infty$, $\norm{\psi_\delta^\pm}_{L^1}$ and $\|\widehat{\psi^\pm_\delta}\|_{\Cc^1}$ are bounded by a constant which only depends on $K$.
\end{lemma}

\section{Admissible functions} \label{sec:admissible-fcns}

In this section, we introduce a class of admissible functions $u$ on $\P^{d-1}$ for which the sequence of random variables $\sigma(S_n,x) + u(S_nx)$ satisfies good limit theorems. Important examples are $u=0$, which will give limit theorems for the norm cocycle, and $u(x) = \log d(x,H_y)$, where $H_y  \subset \P^{d-1}$ is the hyperplane determined by $y \in (\P^{d-1})^*$, which will give limit theorems for the coefficients of $S_n$, see \eqref{eq:coeff-split}.

\begin{definition} \rm \label{def:admissible} 
Let $\eta_*>0$, $0< \alpha_* \leq 1$ and $A_*>0$ be constants. A function  $u:\P^{d-1}\to\R\cup \{\pm \infty\}$ belongs to the class $\Lc(\eta_*, \alpha_*,A_*)$ if $u$ is continuous and
\begin{enumerate}
\item for every $t\geq 0$, we have $\nu\{|u|\geq t\} \leq A_*e^{-\eta_* t}$;
\item for all $x,x'\in \P^{d-1}$ such that $|u(x)|\not=\infty$ and $|u(x')|\not=\infty$, we have 
$$|u(x)-u(x')|\leq A_* d(x,x')^{\alpha_*} \big(e^{\alpha_*|u(x)|}+e^{\alpha_*|u(x')|}\big).$$
\end{enumerate}

We say that $u$ is \textit{admissible} if it belongs to $\Lc(\eta_*, \alpha_*,A_*)$ for some choice of $\eta_*, \alpha_*$ and $A_*$ as above.
\end{definition}
 
For $u$ as above, denote $\Sigma_u:=\{|u|=\infty\}$, which is a compact set. Observe that, from (1), one has that $\nu(\Sigma_u)=0$ and $u \in L^1(\nu)$. Moreover, from (2), we see that $u$ is locally $\alpha_*$-H\"older continuous outside $\Sigma_u$. The class of admissible functions is stable under certain linear combinations and pointwise maxima and minima. This yields many examples of such functions.

The following result can be seen as a version of Benoist-Quint's large deviation estimates, cf. \cite[Lemma 14.11]{benoist-quint:book}. 

\begin{proposition}   \label{prop:LDT-admissible}
Let $u$ be an admissible function in $\Lc(\eta_*, \alpha_*,A_*)$. There exist constants $c>0$ and $n_0, N \in \N$ depending only on $\eta_*, \alpha_*,A_*$ and $\mu$ such that, for $x\in \P^{d-1}$, one has 
	$$   \mu^{*\ell} \big\{g\in G:\, |u(gx)| \geq  m  \big\}  \leq e^{-c m} \quad \text{for} \quad \ell \geq N m \geq n_0, \,\, \ell, m \in \N. $$
\end{proposition}

\begin{proof}
For $m \in \N$, consider the functions $u_m:=\min (|u|,m)$ and $v_m:=  u_m- u_{m-1}$ .  Notice that $0\leq v_m \leq 1$, $v_m=0$ on $\big\{|u|\leq m -1 \big\}$, and $v_m=1$ on $\big\{|u|\geq m \big\}$. Then,
	$$\mu^{*\ell} \big\{g\in G:\, |u(gx)| \geq   m  \big\} = \oP^\ell \mathbf 1_{|u|\geq m} (x) \leq \oP^\ell \big( v_m \big)(x).$$
	Observe that, from Property (2) of Definition \ref{def:admissible}, the restriction of $u$ to $\{|u| \leq m\}$ is $\alpha_*$-H\"older continuous, with H\"older constant bounded by $2A_* e^{\alpha_* m}$. Hence, $\norm{u_m}_{\Cc^{\alpha_*}} \lesssim e^{\alpha_* m}$ because the functions $|\cdot|$ and $\min$ are Lipschitz. It follows that $\norm{v_m}_{\Cc^{\alpha_*}} \lesssim e^{\alpha_* m}$ and, by interpolation, $\norm{v_m}_{\Cc^{\alpha}} \lesssim e^{\alpha m}$ for all $0<\alpha\leq \alpha_*$, see \cite{triebel}.
	
Fix $0<\alpha \leq \alpha_*$ small enough. From Theorem \ref{thm:spectral-gap}, $\norm{\oP^\ell - \oN}_{\Cc^\alpha} \lesssim \tau^\ell$ for some  $0<\tau<1$. We can write $\tau = e^{-\beta}$, where $\beta >0$. Therefore, for $m$ large enough and $\ell \geq N m$ for some fixed large $N \geq 1$, one has
	$$ \oP^\ell \big( v_m \big)(x) \lesssim  \lp \nu, v_m \rp + \norm{v_m}_{\Cc^{\alpha}}  \tau^\ell \lesssim \lp \nu, v_m \rp + e^{\alpha m} e^{- \beta \ell}  \leq \lp \nu, v_m \rp + e^{(\alpha - N \beta) m }  \lesssim \lp \nu, v_m \rp + e^{-c_1 m} $$ 
 for some constant $c_1>0$. Now, from Property (2) of Definition \ref{def:admissible}, we have that
	\begin{equation*} 
	\lp \nu, v_m \rp\leq \nu\big\{|u|\geq m -1\big\} \leq A_*e^{-\eta_* (m-1)} \lesssim e^{- c_2 m},
	\end{equation*}
	for some constant $c_2>0$. The lemma follows.
\end{proof}

\begin{corollary} \label{cor:LDT-admissible}
There exist constants $A>0$ and $C$ depending only on $\eta_*, \alpha_*,A_*$ and $\mu$ such that, for all $\Lc(\eta_*, \alpha_*,A_*)$, one has
\begin{equation*}
\mu^{*n} \big\{g\in G:\,  |u(gx)| \geq  A \log n \big\} \leq C / n.
\end{equation*}
\end{corollary}

\begin{proof}
It is enough to consider $n$ large. Taking $\ell:=n $ and $m:=\lfloor A \log n \rfloor$ in Proposition \ref{prop:LDT-admissible} for some large $A>0$  gives
\begin{equation*} 
\mu^{*n} \big\{g\in G:\,  |u(gx)| \geq  A \log n \big\} \leq e^{-c\lfloor A\log n\rfloor} \leq e^ c n^{-cA} \leq C / n, 
\end{equation*}
for some constant $C>0$ since $A$ is large. The result follows.
\end{proof}

We now construct a partition of unity adapted to $u$. This is inspired by the construction in \cite{DKW:BE-LLT-coeff} and will be crucial in the proofs of our main theorems.

For $u \in \Lc(\eta_*, \alpha_*,A_*)$ and $k\in \Z$, define $$\Tc^u_k:=\big\{w \in \P^{d-1}: |u(w)+k|<1 \big\}.$$

\begin{lemma} \label{lemma:partition-of-unity}
	There exist non-negative smooth functions $\chi_k$ on $\P^{d-1}$, $k\in\Z$, 
	 such that $0 \leq \chi_k \leq 1$ and 
	\begin{enumerate}
		\item $\chi_k$ is supported by $\Tc^u_k$;
		\item If $w \in \P^{d-1} \setminus \Sigma_u$, then  $\chi_k(w) \neq 0$ for at most two values of $k$;
		\item $\sum_{k\in\Z}  \chi_k=1$ on $\P^{d-1} \setminus \Sigma_u$; 
		\item For every exponent $0<\alpha\leq \alpha_*$, one has $\norm{\chi_k}_{\Cc^{\alpha}}\leq C_\alpha e^{\alpha |k|}$, where $C_\alpha>0$ is a constant depending on $\eta_*, \alpha_*,A_*$ and $\alpha$ but independent of $u$ (for $\alpha=\alpha_*=1$, we replace $\norm{ \cdot }_{\Cc^1}$ by $\norm{ \cdot }_{\Lip}$).
	\end{enumerate}
\end{lemma}
\proof
It is easy to find a smooth function $0 \leq \widetilde \chi \leq 1$ supported by $(-1,1)$ such that $\widetilde \chi(t) = 1$ for $|t|$ small, $\widetilde \chi(t) + \widetilde \chi(t-1)= 1$ for $0 \leq t\leq 1$ and $\norm{\widetilde \chi}_{\Cc^1}\leq 4$.  Define  $\widetilde \chi_k (t) := \widetilde \chi(t+k)$. We see that $\widetilde \chi_k$ is supported by $(-k-1,-k+1)$,  $\sum_{k\in\Z} \widetilde \chi_k=1$ on $\R$ and $\norm{\widetilde \chi_k}_{\Cc^1}\leq 4$. Set $\chi_k(w):= \widetilde \chi_k(u(w))$. Assertion (1)--(3) are clear from the definition, so we only need to prove the last estimate. By interpolation theory \cite{triebel}, we only need to consider the case $\alpha=\alpha_*$. We need to show that, for $x,x'\in \P^{d-1}$, one has that
\begin{equation} \label{eq:chi_k-holder}
|\chi_k(x)-\chi_k(x')|\lesssim e^{|k|\alpha_*} d(x,x')^{\alpha_*}.
\end{equation}

If $\chi_k$ vanishes at $x$ and $x'$ there is nothing to prove, so we can assume that $\chi_k(x) \neq 0$. In particular, we have that $|u(x)+k|<1$. 

\medskip

\noindent
{\bf Claim.} There exists a small constant $q>0$ such that $|u(w)+k|\leq 2$ when $d(x,w)\leq q e^{-|k|}$. 

\medskip

\noindent {\it Proof of Claim.} If the claim is not true, then, by continuity, we can find $z \in \P^{d-1}$ such that $|u(z)+k|=2$ and  $d(x,z)\leq q e^{-|k|}$.
Together with the fact $|u(x)+k|<1$, this contradicts Property (2) of Definition \ref{def:admissible} when $q$ is small. \endproof

We now verify \eqref{eq:chi_k-holder}.

\medskip	

	\noindent\textbf{Case 1.} $d(x,x')> q e^{-|k|}$. In this case, $e^{|k|\alpha_*} d(x,x')^{\alpha_*} \geq q^{\alpha_*}$. Since $|\chi_k|\leq 1$, one has that $$|\chi_k(x)-\chi_k(x')| \leq 2 = 2 q^{-\alpha_*} q^{\alpha_*} \leq 2 q^{-\alpha_*} e^{\alpha_* |k|} d(x,x')^{\alpha_*},$$
	thus giving \eqref{eq:chi_k-holder}.

\medskip
	
	\noindent\textbf{Case 2.} $d(x,x') \leq q e^{-|k|}$. From the Claim and the fact that $|u(x)+k|<1$ we have that $e^{\alpha_*|u(x)|}+e^{\alpha_*|u(x')|} \lesssim e^{\alpha_* |k|}$. Together with Property (2) of Definition \ref{def:admissible},  we obtain that
$$|u(x)-u(x')|\lesssim e^{\alpha_* |k|} d(x,x')^{\alpha_*}.$$
Since $\chi_k = \widetilde \chi_k \circ u$ and $\norm{\widetilde \chi_k}_{\Cc^1}\leq 4$, the above estimate  gives \eqref{eq:chi_k-holder}.

We have thus obtained (4), finishing the proof of the lemma.
\endproof

Observe that $u = 0$ is admissible in the sense of Definition \ref{def:admissible}. In this case $\Tc^u_0 = \P^{d-1}$, $\Tc^u_k = \varnothing$ for $k \neq 0$ and the above partition is given by $\chi_0 = \mathbf 1$ and $\chi_k \equiv 0$ for $k \neq 0$.

\begin{lemma} \label{lemma:log-dist-admissible}
Let $\eta$ be the constant in Proposition \ref{prop:regularity}. Then, there is a constant $A_*>0$ such that for every $y\in (\P^{d-1})^*$ the function $\log d(x,H_y)$ belongs to $\Lc(\eta,1, A_*)$. 
\end{lemma}
\proof
Fix a constant $A_*>0$ large enough. Denote by $u(x):=\log d(x,H_y)$. Then, $\Sigma_u=H_y$. 
Property (1) in the definition of $\Lc(\eta,1, A_*)$ follows directly from Proposition \ref{prop:regularity}. We now verify Property (2) with $\alpha_* = 1$.  Notice that, for fixed $0<r_0<1$, $u$ is Lipshcitz on $\P^{d-1} \setminus \B(H_y,r_0)$. Hence, we can assume that $x$ and $x'$ are close to $H_y$. We can assume that $d(x,H_y) < d(x',H_y)$. Denote by $[x',x]$ the arc between $x'$ and $x$ along the projective line between them. Using that  $\diff := d( \cdot , H_y)$ is $1$-Lipchitz,  $u = \ell \circ \diff$  and that the function $\ell(t) := \log t$ satisfies $\ell'(t) = e^{|\ell(t)|} $ for $0< t \leq 1$, we have
\begin{align*}
|u(x)-u(x')|  &\leq \sup_{w \in [x',x]} e^{|\ell(\diff(w))|} | \diff(x) - \diff(x')| \leq d(x,x') \sup_{w \in [x',x]} e^{|u(w)|} \\ &\leq d(x,x') e^{|u(x)|} \leq d(x,x')  \big(e^{|u(x)|}+e^{|u(x')|}\big).
\end{align*}
This ends the proof.
\endproof

%%%%%%%%%%%%
\section{Berry-Esseen bound} \label{sec:BE}

This section is devoted to the proof of Theorem \ref{thm:BE-general}.  We start with a version Berry-Esseen lemma, obtained in \cite[Corollary 3.2]{DKW:BE-LLT-coeff}, that is necessary to us, see also \cite[XVI.3]{feller:book}. As mentioned in the introduction, this allows us to avoid the integration of the characteristic functions on complex contours, which is a main technical difficulty in \cite{xiao-grama-liu}. 

 For a real random variable $X$ with cumulative distribution function $F$ (c.d.f.\ for short), we let $$\phi_F(\xi):=\mathbf E\big(e^{-i\xi X}\big)$$
be its \textit{conjugate characteristic function}.

\begin{lemma}  \label{lemma:BE-feller} 
	Let $F$ be a c.d.f.\ of some real random variable and let  $H$  be a differentiable real-valued function  on $\R$ with derivative $h$ such that $H(-\infty)=0,H(\infty)=1,|h(t)|\leq m$ for some constant $m >0$. Let $D>0$ and $0<\delta<1$ be real numbers such that $\big|F(t)-H(t) \big|\leq D \delta^2$ for $|t|\geq \delta^{-2}$. 
	Assume moreover that $h\in L^1$, $\widehat h\in\Cc^1$ and that $\phi_F$ is differentiable at zero.
	Then, there exist a constant $C>0$ (resp. $\kappa > 1$) independent of $F,H,\delta$ (resp. $D, F,H,\delta$), such that
\begin{eqnarray*}
\sup_{t\in\R}\big|F(t)-H(t) \big| &\leq & 
	{1\over \pi} \sup_{|t|\leq \kappa \delta^{-2}}     \Big|\int_{-\delta^{-2}}^{\delta^{-2}} {\Theta_t(\xi) \over \xi}     \,\diff \xi\Big|  +C\delta^2 \\
	&=& 
	{1\over \pi} \sup_{|t|\leq \kappa \delta^{-2}}    \Big|\int_{0}^{\delta^{-2}} {\Theta_t(\xi)-\Theta_t(-\xi) \over \xi}     \,\diff \xi\Big|  +C\delta^2,
\end{eqnarray*}
	where  $\Theta_t(\xi):=e^{it\xi}\big(\phi_F(\xi)- \widehat {h}(\xi) \big)\widehat{\vartheta_\delta}(\xi)$,
	and $\vartheta_\delta$ is defined in Lemma \ref{lemma:vartheta}.
\end{lemma}

The last expression for the integral in the above lemma will be used when dealing with the singularity of the integrand at the origin. In fact, that is so called ``Cauchy principal value", which is equal to the  first integral  since ${\Theta_t(\xi) \over \xi}$ is bounded.

\medskip

We now begin the proof of Theorem \ref{thm:BE-general}. Observe that it  suffices to prove the theorem for $J=(-\infty,b]$, as the case of an arbitrary interval can be obtained by taking complements and intersections. In order to simplify the notation, for fixed $x\in \P^{d-1}$, we define two functionals
\begin{equation} \label{eq:En-def}
\oE_n (\Phi) :=\mathbf E \Big(  \Phi\Big( {\sigma(S_n, x) + u(S_n x) - n \gamma \over \sqrt n}, S_n x  \Big)\Big)
\end{equation}
and 
$$\oR(\Phi):= \frac{1}{\sqrt{2 \pi}  \,\varrho} \int_{\P^{d-1}} \int_{-\infty}^\infty e^{-\frac{s^2}{2 \varrho^2}} \Phi(s,w)\, \diff s\diff\nu(w),$$
where $\Phi(t,w)$ is a real-valued function on $\R\times \P^{d-1}$.  Then, the conclusion of Theorem \ref{thm:BE-general} for $J=(-\infty,b]$ corresponds to 
\begin{equation} \label{eq:E_n-R}
\big|  \oE_n ( \mathbf 1_{(-\infty,b]} \psi \cdot \varphi) - \oR( \mathbf 1_{(-\infty,b]}  \psi \cdot \varphi)  \big| \leq \frac{C}{\sqrt n}\norm{\psi}_\Hc\norm{\varphi}_{\Cc^\alpha}.
\end{equation}

\subsection{The case $\psi= \mathbf 1$}

We first prove Theorem \ref{thm:BE-general} in the case where $\psi= \mathbf 1$. We note that when $u=0$, this was obtained in \cite{xiao-grama-liu} and when $u(x) = \log d(x,H_y)$ this is treated in \cite{DKW:BE-LLT-coeff,xiao-grama-liu:coeff} and was considered in the first version of this article.  The version for general admissible $u$ is new.

\begin{proposition}\label{prop:BE-coeff-1-varphi}
	Theorem \ref{thm:BE-general} holds for $\psi=\mathbf 1$.
\end{proposition}

We now begin the proof of Proposition \ref{prop:BE-coeff-1-varphi}. We can assume that  ${1\over 2}\leq \varphi\leq 2$,  $\oN_0 \varphi=1$ and $\norm{\varphi}_{\Cc^\alpha}\leq 2$, since the problem is linear on $\varphi$ and these functions span the space $\Cc^\alpha(\P^{d-1})$. 
\vskip 5pt

Fix a large constant $A>1$. From Corollary \ref{cor:LDT-admissible},
\begin{equation} \label{A-large-ineq}
\mu^{*n} \big\{g\in G:\,  |u(gx)| \geq  A \log n \big\} \lesssim 1 /\sqrt n.
\end{equation}
In other words, $|u(S_n x)| \geq  A \log n$ with probability smaller a constant times $1 /\sqrt n$. In particular, \eqref{eq:E_n-R} for $\psi=\mathbf 1$ will follow if we show that
\begin{equation}\label{goal-1-varphi}
\big| \oL_n^b(\varphi)  -  H(b) \big| \lesssim \frac{1}{\sqrt n},
\end{equation}
where
$$\oL_n^b(\varphi):=  \mathbf E \Big(  \mathbf 1_{{\sigma(S_n,x) + u(S_n x) - n \gamma \over \sqrt n}\leq b}   \mathbf 1_{  |u(S_n x)| \leq  A \log n  } \,  \varphi(S_n x)\Big)$$
and
$$H(b):=\oR\big(\mathbf 1_{(-\infty,b]}\varphi\big) = \frac{1}{\sqrt{2 \pi} \, \varrho} \int_{\P^{d-1}} \varphi\,\diff\nu \int_{-\infty}^b e^{-\frac{t^2}{2 \varrho^2}} \, \diff t=\frac{1}{\sqrt{2 \pi}  \,\varrho}  \int_{-\infty}^b e^{-\frac{t^2}{2 \varrho^2}} \, \diff t$$
is  the c.d.f.\ of the normal distribution $\cali N(0;\varrho^2)$. Recall that $\varphi$ is bounded and $\oN_0\varphi= \int_{\P^{d-1}} \varphi \, \diff \nu = 1$.   

Let $\chi_k$, $k \in \Z$, be the functions from Lemma \ref{lemma:partition-of-unity}. Then, for $w\in\P^{d-1}$,
$$ \sum_{|k| \leq A\log n-1}(\chi_k\varphi)(w)\leq \mathbf 1_{  |u(w)| \leq  A \log n  } \cdot \varphi(w) \leq  \sum_{|k| \leq A\log n+1}(\chi_k\varphi)(w). $$
Moreover, $\chi_k(w)$ is non-zero only when $-k-1 < u(w) < -k+1$. Thus, 
\begin{align}
\sum_{|k| \leq A\log n-1} \mathbf E\Big(  \mathbf 1_{{\sigma(S_n,x) - n \gamma  - k+1\over \sqrt n}\leq b} \, (\chi_k\varphi)(S_n x) \Big)&\leq \oL_n^b(\varphi)  \nonumber\\
&\leq \sum_{|k| \leq A\log n+1} \mathbf E\Big(  \mathbf 1_{{\sigma(S_n,x) - n \gamma  - k-1\over \sqrt n}\leq b} \, (\chi_k\varphi)(S_n x) \Big). \label{eq:Ln-two-sided-bound}
\end{align}

Define the functions  $\Phi_{n}^{\star}$ and $\Phi_{n,\xi}$ on $\P^{d-1}$ by
\begin{equation} \label{eq:Phi-star-def}
\Phi_{n}^{\star}  (w):= \varphi(w) - \sum_{|k| \leq A\log n} (\chi_k\varphi)(w)
\end{equation}
and
\begin{equation} \label{eq:Phi-xi-def}
\Phi_{n,\xi} (w):= \sum_{|k| \leq A\log n}  e^{i \xi{k \over \sqrt n}}(\chi_k\varphi)(w),
\end{equation}
where $n\geq 1$ and $\xi\in\R$.

Let $d_{n,x}:=\big(  \oP_0^n\varphi(x) \big)^{-1}$, which will be used as a normalization factor. Recall that ${1\over 2}\leq \varphi\leq 2$ by assumption.  By  Proposition \ref{prop:spectral-decomp}-(5), there exist constants $c>0$ and $0<\beta<1$ such that 
$\|\oQ_0^n \|_{\Cc^\alpha}\leq c\beta^n$. Therefore, 
$$\big|\oP_0^n\varphi(x)-\oN_0  \varphi\big| =\big|\oQ_0^n \varphi(x)\big|\leq c\beta^n \norm{\varphi}_{\Cc^\alpha} \leq 2c \beta^n,$$ so, since $\oN_0 \varphi =1$,
$${1\over 2}\leq d_{n,x}\leq 2  \quad\text{and}\quad |d_{n,x} -1| \leq 4c\beta^n \quad\text{for } n \text{ large enough}.$$

For $n \geq 1$, let
$$F_n(b):=  d_{n,x} \sum_{|k| \leq A\log n} \mathbf E\Big(  \mathbf 1_{{\sigma(S_n,x) - n \gamma  - k\over \sqrt n}\leq b} \, (\chi_k \varphi)(S_n x) \Big) +d_{n,x}\mathbf E\Big(  \mathbf 1_{{\sigma(S_n,x) - n \gamma  \over \sqrt n}\leq b} \, \Phi_{n}^{\star}(S_n x) \Big).$$
Then, $F_n$ is a  non-decreasing, right-continuous functions with $F_n(-\infty)=0$ and
$$F_n(\infty)= d_{n,x} \E(\varphi(S_n x)) = d_{n,x} \oP_0^n \varphi(x) =1.$$
Therefore, it is the c.d.f.\ of some probability distribution.

Recall that $H(b)$ is the c.d.f.\ of the normal distribution $\cali N(0;\varrho^2)$. Then, $H(-\infty)=0,H(\infty)=1$ and the derivative $h$ of $H$ satisfies $h(b)={1\over \sqrt{2\pi}\rho} e^{-{b^2 \over 2\rho^2}}$ and  $\widehat h(\xi)=e^{-{\varrho^2 \xi^2\over 2}}$. In particular, $h$ is bounded by $1/(\sqrt {2\pi}\varrho)$.

The following results correspond to Lemmas 3.4 to 3.7 in \cite{DKW:BE-LLT-coeff} when $\varphi = \mathbf 1$ and $u(x) = \log d(x,H_y)$. The proofs given there can be easily adapted to our case. In particular, the usual large deviation estimates for $d(S_n x,H_y)$ used there should be replaced by Corollary \ref{cor:LDT-admissible}. Recall that ${1\over 2}\leq \varphi\leq 2$ by assumption.

\begin{lemma} \label{lemma:Ln-Fn}
	Let $\oL_n^b$ and $F_n$ be as above. Then, there exists a constant $C>0$ independent of $n$, such that for all $n \geq 1$ and $b \in \R$, 
	$$   F_n\Big(b - {1 \over \sqrt n}\Big)  - {C \over \sqrt n} \leq \oL_n^b(\varphi) \leq   F_n\Big(b + {1 \over \sqrt n}\Big)  + {C \over \sqrt n}.$$
\end{lemma}

\begin{lemma} \label{lemma:char-function-Fn}
	The conjugate characteristic function of $F_n$ is given by
	$$\phi_{F_n}(\xi)=  d_{n,x} e^{i\xi\sqrt n \gamma} \oP_{-{i\xi\over \sqrt n}}^n \big( \Phi_{n,\xi} +  \Phi_{n}^{\star} \big)  (x).$$
	In particular,  $\phi_{  F_n}$ is differentiable near zero.
\end{lemma}

\begin{lemma}  \label{lemma:norm-Phi}
	Let $\Phi_{n,\xi}, \Phi_{n}^{\star}$ be the functions on $\P^{d-1}$ defined above.  Then,  the following identity holds
	\begin{equation} \label{eq:psi_xi+psi_T}
	\Phi_{n,\xi} + \Phi_{n}^{\star} = \mathbf \varphi + \sum_{|k| \leq A\log n} \big(  e^{i \xi{k \over \sqrt n}} - 1 \big) \chi_k\varphi.
	\end{equation}
	Moreover,  for every exponent $0<\alpha\leq \alpha_*$, we have $\norm{\chi_k\varphi}_{\Cc^\alpha}\lesssim e^{\alpha |k|}\norm{\varphi}_{\Cc^\alpha}$ and there  is a constant $C>0$ independent of $\xi$ and $n$ such that   
	\begin{equation}  \label{eq:norm-Phi}
	\norm{\Phi_{n,\xi} }_{\Cc^\alpha}\leq C \,  n^{\alpha A} \quad\text{and}\quad  \norm{\Phi_{n}^{\star}  }_{\Cc^\alpha} \leq C \,  n^{\alpha A}.
	\end{equation}
	In addition,  $\Phi_{n}^{\star}  $ is supported by $\big\{|u| \geq A\log n -1\big\}$.
\end{lemma}

Let $\xi_0$ be a small constant satisfying Lemma \ref{lemma:lambda-estimates}.

\begin{lemma} \label{lemma:F_n-estimate}
	Let $F_n$ and $H$ be as above. Then, $\big| F_n(b)-H(b)\big|\lesssim 1/\sqrt n$ for $|b|\geq \xi_0 \sqrt n$.
\end{lemma}

From now on, we fix $0 < \alpha <1$ small enough so that
$$\alpha A\leq 1/6 \quad \text{ and } \quad \alpha \leq \alpha_0,$$ where $0<\alpha_0<1$ is the exponent appearing in Theorem \ref{thm:spectral-gap}. The reason for this choice will be more clear below. From the results of Subsection \ref{subsec:markov-op}, the family $\xi\mapsto\oP_{i\xi}$ acts continuously on $\Cc^\alpha(\P^{d-1})$ for $\xi \in \R$, it is analytic near $0$ and $\oP_0$ has a spectral gap. 

Lemmas \ref{lemma:char-function-Fn} and \ref{lemma:F_n-estimate}  imply that $  F_n$ and $H$ satisfy the conditions of Lemma \ref{lemma:BE-feller} with $\delta_n:=(\xi_0 \sqrt n)^{-1/2}$.  Let $\kappa > 1$ be the constant appearing in that lemma. For simplicity, by taking a smaller $\xi_0$ is necessary, one can assume that $2\kappa \xi_0 \leq 1$.  Then,  Lemma \ref{lemma:BE-feller} gives that
\begin{equation} \label{eq:BE-main-estimate}
\sup_{b \in\R}\big|F_n(b)-H(b)\big|\leq  {1\over \pi} \sup_{|b|\leq \sqrt n}    \Big|  \int_{0}^{\xi_0\sqrt n} {\Theta_b(\xi) -\Theta_b(-\xi)\over \xi}    \,\diff \xi  \Big|+{C\over \sqrt n},
\end{equation}
where $C>0$ is a constant independent of $n$ and 
$$\Theta_b(\xi):=e^{ib\xi}\big( \phi_{  F_n}(\xi)- \widehat {h}(\xi) \big)\widehat{\vartheta_{\delta_n}}(\xi).$$

\medskip

We now estimate the integral in  \eqref{eq:BE-main-estimate}. In order to deal with $\phi_{F_n}$, we'll use Lemmas \ref{lemma:char-function-Fn} coupled with Proposition \ref{prop:spectral-decomp} and Lemma \ref{lemma:lambda-estimates}. Recall that $H(b)$ is the c.d.f.\ of the normal distribution $\cali N(0;\varrho^2)$ and, in the notation of Lemma \ref{lemma:BE-feller}, we have $h(s): =\frac{1}{\sqrt{2 \pi}  \,\varrho}   e^{-\frac{s^2}{2 \varrho^2}}$ and $\widehat h(\xi) =  e^{-{\varrho^2 \xi^2\over 2}}$. It is then natural to introduce the function
$$\widetilde h_n(\xi):= d_{n,x} e^{-{\varrho^2 \xi^2\over 2}}  \oN_0 (  \Phi_{n,\xi} +  \Phi_{n}^{\star} ).$$ Notice that for every fixed $n$ and $\xi$,  $\oN_0 (  \Phi_{n,\xi} +  \Phi_{n}^{\star} )$ is a constant independent of $x$.  Define also
$$\Theta_b^{(1)}(\xi):= e^{ib\xi}\big( \phi_{F_n}(\xi)- \widetilde {h}_n(\xi) \big)\widehat{\vartheta_{\delta_n}}(\xi) \quad\text{and} 
\quad \Theta_b^{(2)}(\xi):= e^{ib\xi}\big(  \widetilde {h}_n(\xi)-\widehat h(\xi) \big)\widehat{\vartheta_{\delta_n}}(\xi),$$
so that $\Theta_b =  \Theta_b^{(1)} + \Theta_b^{(2)}$.

\begin{lemma} \label{lemma:theta-1-bound} 
	We have
	$$\sup_{|b|\leq \sqrt n}    \Big|  \int_{0}^{\xi_0\sqrt n} {\Theta_b^{(1)}(\xi) -\Theta_b^{(1)}(-\xi)\over \xi}    \,\diff \xi  \Big| \lesssim {1 \over \sqrt n}.$$
\end{lemma}

\begin{proof}
	Using Lemma \ref{lemma:char-function-Fn}, the decomposition of $\oP_z$ from Proposition \ref{prop:spectral-decomp},  we have
	$$\Theta_b^{(1)}=d_{n,x}\big(\Lambda_1+\Lambda_2+\Lambda_3\big),$$
	where 
	$$ \Lambda_1(\xi):= e^{ib\xi}\Big(  e^{i\xi\sqrt n \gamma} \lambda_{-{i\xi\over \sqrt n}}^n \oN_{-{i\xi\over \sqrt n}} (\Phi_{n,\xi} +\Phi_{n}^{\star}  ) (x) -      e^{-{\varrho^2 \xi^2\over 2}}   \oN_0 (\Phi_{n,\xi} +\Phi_{n}^{\star}  )  (x)   \Big) \widehat{\vartheta_{\delta_n}}(\xi), $$ 
	$$\Lambda_2(\xi):=e^{ib\xi}\Big(  e^{i\xi\sqrt n \gamma} \oQ_{-{i\xi\over \sqrt n}}^n  (\Phi_{n,\xi} +\Phi_{n}^{\star}  )(x) -e^{i\xi\sqrt n \gamma}\oQ_0^n  (\Phi_{n,\xi} +\Phi_{n}^{\star}  )(x)  \Big)\widehat{\vartheta_{\delta_n}}(\xi)$$
	and
	$$ \Lambda_3(\xi):=e^{ib\xi}e^{i\xi\sqrt n \gamma}\oQ_0^n  (\Phi_{n,\xi} +\Phi_{n}^{\star}  )(x)  \,\widehat{\vartheta_{\delta_n}}(\xi).$$

Recall that $\frac12 \leq d_{n,x} \leq2$, so it enough to estimate the integrals of	${\Lambda_j(\xi) -\Lambda_j(-\xi)\over \xi}$ for $j=1,2,3$. 	Notice that $\Lambda_1(0)=\Lambda_2(0)=0$, so  for $j=1,2$, we have 
	$$\int_{0}^{\xi_0\sqrt n} {\Lambda_j(\xi) -\Lambda_j(-\xi)\over \xi}    \,\diff \xi = \int_{-\xi_0\sqrt n}^{\xi_0\sqrt n} {\Lambda_j(\xi)\over \xi}    \,\diff \xi.$$

	In order to estimate $\Lambda_2$, observe that, for $z$ small, the norm of the operator $\oQ^n_z - \oQ^n_0$ is bounded by a constant times $|z| n \beta^n$ for some $0<\beta<1$. This can be seen by writing the last difference as $\sum_{\ell=0}^{n-1}  \oQ_z^{n-\ell-1}(\oQ_z - \oQ_0)\oQ_0^\ell$, applying Proposition \ref{prop:spectral-decomp}-(5) and using the facts that $\norm{\oQ_z}_{\Cc^\alpha} \lesssim \beta^n$ and $\norm{\oQ_z - \oQ_0}_{\Cc^\alpha} \lesssim |z|$.   Therefore, we have 
	$$  \Big|  \oQ_{-{i\xi\over \sqrt n}}^n  (\Phi_{n,\xi} +\Phi_{n}^{\star}  )(x) - \oQ_0^n (\Phi_{n,\xi} +\Phi_{n}^{\star}  )(x) \Big| \lesssim {|\xi|\over \sqrt n} n \beta^n \norm{ \Phi_{n,\xi} +\Phi_{n}^{\star}  }_{\Cc^\alpha}.$$
	Using \eqref{eq:norm-Phi}, we get that
	\begin{align*}
	\Big|\int_{-\xi_0\sqrt n}^{\xi_0\sqrt n} {\Lambda_2(\xi) \over \xi}    \,\diff \xi\Big|
	&\leq\int_{-\xi_0\sqrt n}^{\xi_0\sqrt n} {1\over |\xi|} \cdot \Big| \oQ_{-{i\xi\over \sqrt n}}^n  (\Phi_{n,\xi} +\Phi_{n}^{\star}  )(x) -  \oQ_0^n (\Phi_{n,\xi} +\Phi_{n}^{\star}  )(x) \Big|\,\diff\xi \\
	&\lesssim     \int_{-\xi_0\sqrt n}^{\xi_0\sqrt n} {1\over |\xi|} \cdot  |\xi| \sqrt n \beta^n \norm{ \Phi_{n,\xi} +\Phi_{n}^{\star}  }_{\Cc^\alpha} \,\diff\xi \lesssim \beta^n n^{\alpha A + 1} \lesssim { 1 \over \sqrt n}.     
	\end{align*}
	\vskip 3pt
	
	We now estimate $\Lambda_3$ using its derivative $\Lambda'_3$. Recall that $|\widehat{\vartheta_{\delta_n}}|\leq 1,\norm{\widehat{\vartheta_{\delta_n}}}_{\Cc^1}\lesssim 1,|b|\leq \sqrt n$ and  $\big|\oQ_0^n  (\Phi_{n,\xi} +\Phi_{n}^{\star}  )(x) \big|\lesssim \beta^n\norm{ \Phi_{n,\xi} +\Phi_{n}^{\star}  }_{\Cc^\alpha}$, where $0<\beta<1$ is as before. A direct computation using the definition of $\Phi_{n,\xi} $ gives for $|\xi|\leq \xi_0\sqrt n$,
	\begin{align*}
	\big| \Lambda_3'(\xi)\big| &\leq  \Big|b \oQ_0^n  (\Phi_{n,\xi} +\Phi_{n}^{\star}  )(x)\Big|+    \Big|\sqrt n \gamma \oQ_0^n  (\Phi_{n,\xi} +\Phi_{n}^{\star}  )(x)\Big|  \\ 
	& \quad\quad \quad\quad \quad\quad \quad + \sum_{|k| \leq A\log n} \Big|{k\over \sqrt n} e^{i \xi \frac{k}{\sqrt n}}\oQ_0^n (\chi_k\varphi) (x)    \Big|  + \Big|\oQ_0^n  (\Phi_{n,\xi} +\Phi_{n}^{\star}  )(x)\Big|\cdot \norm{\widehat{\vartheta_{\delta_n}}}_{\Cc^1}\\
	&\lesssim \sqrt n  \beta^n n^{\alpha A}  +\sum_{|k| \leq A\log n} {|k| \over \sqrt n} \beta^n n^{\alpha A}  +   \beta^n n^{\alpha A}   \\ &\lesssim \sqrt n  \beta^n n^{\alpha A}   + \frac{(\log n)^2}{\sqrt n} \beta^n n^{\alpha A}  +   \beta^n n^{\alpha A}    \lesssim (1+\sqrt n)  \beta^n n^{\alpha A},
	\end{align*}
	where in the second inequality we have used that $\norm{\chi_k\varphi}_{\Cc^\alpha} \lesssim e^{\alpha |k|} \norm{\varphi}_{\Cc^\alpha} \lesssim \, n^{\alpha A} $ and $\norm{ \Phi_{n,\xi} +\Phi_{n}^{\star}  }_{\Cc^\alpha} \lesssim n^{\alpha A}$, see Lemma \ref{lemma:norm-Phi}.
	Applying the mean value theorem to $\Lambda_3$ between $\xi$ and $-\xi$, yields
	\begin{align*}
	\Big|\int_{0}^{\xi_0\sqrt n} {\Lambda_3(\xi)-\Lambda_3(-\xi) \over \xi}    \,\diff \xi\Big|\leq 2\xi_0\sqrt n\sup_{|\xi|\leq \xi_0\sqrt n} \big| \Lambda_3'(\xi) \big|\lesssim \sqrt n (1+\sqrt n)  \beta^n n^{\alpha A}   \lesssim {1 \over \sqrt n}.
	\end{align*}
	\vskip 5pt
	
	It remains to estimate the term involving $\Lambda_1$. We have
	$$ \Big|\int_{-\xi_0\sqrt n}^{\xi_0\sqrt n} {\Lambda_1(\xi) \over \xi}   \diff \xi\Big| \leq \int_{-\xi_0\sqrt n}^{\xi_0\sqrt n} {1\over |\xi|} \cdot \Big|  e^{i\xi\sqrt n \gamma} \lambda_{-{i\xi\over \sqrt n}}^n \oN_{-{i\xi\over \sqrt n}} (\Phi_{n,\xi} +\Phi_{n}^{\star}  ) (x) -    e^{-{\varrho^2 \xi^2\over 2}}   \oN_0 (\Phi_{n,\xi} +\Phi_{n}^{\star}  )(x)    \Big|  \diff \xi.$$
	We split the last integral into two integrals using
	$$ \Gamma_1(\xi):=   e^{i\xi\sqrt n \gamma} \lambda_{-{i\xi\over \sqrt n}}^n \oN_{-{i\xi\over \sqrt n}}  (\Phi_{n,\xi} +\Phi_{n}^{\star}  ) (x) -      e^{i\xi\sqrt n \gamma} \lambda_{-{i\xi\over \sqrt n}}^n  \oN_0 (\Phi_{n,\xi} +\Phi_{n}^{\star}  )(x)   $$
	and
	$$  \Gamma_2(\xi):=   e^{i\xi\sqrt n \gamma} \lambda_{-{i\xi\over \sqrt n}}^n  \oN_0 (\Phi_{n,\xi} +\Phi_{n}^{\star}  ) (x) -      e^{-{\varrho^2 \xi^2\over 2}}   \oN_0 (\Phi_{n,\xi} +\Phi_{n}^{\star}  ) (x) .   $$
	\vskip 5pt
	
	\noindent\textbf{Case 1.}  $\sqrt[6] n<|\xi|\leq \xi_0\sqrt n$.    In this case, by Lemma \ref{lemma:lambda-estimates}, we have
	\begin{equation} \label{eq:case1-lambda}
	\big|\lambda_{-{i\xi\over \sqrt n}}^n\big|\leq e^{-{\varrho^2\xi^2\over 3}} \quad \text{and}\quad \Big| e^{i\xi\sqrt n \gamma} \lambda_{-{i\xi\over \sqrt n}}^n-e ^{-{\varrho^2\xi^2\over 2}}  \Big|\lesssim {1\over \sqrt n}e^{-{\varrho^2\xi^2\over 4}}.
	\end{equation}
	
	From the analyticity of $\xi\mapsto \oN_{i\xi}$ (cf. Proposition \ref{prop:spectral-decomp}), Lemma \ref{lemma:norm-Phi} and the fact that $\alpha A\leq 1/6$, one has 
	$$\Big\|\big(\oN_{-{i\xi\over \sqrt n}}-\oN_0 \big)  (\Phi_{n,\xi} +\Phi_{n}^{\star}  ) \Big\|_\infty \lesssim {|\xi| \over \sqrt n} \norm{ \Phi_{n,\xi} +\Phi_{n}^{\star}  }_{\Cc^\alpha}\lesssim {|\xi| \over \sqrt n}n^{\alpha A}   \leq  {|\xi| \over \sqrt n}\sqrt[6] n  . $$
	Hence, using \eqref{eq:case1-lambda}, we get
	$$\int_{\sqrt[6] n<|\xi|\leq \xi_0\sqrt n} {1\over |\xi|}\cdot \big| \Gamma_1(\xi) \big|  \,\diff \xi \lesssim \int_{- \infty}^{\infty} {1\over \sqrt[6] n} \cdot e^{-{\varrho^2\xi^2\over 3}} {|\xi| \over \sqrt n}\sqrt[6] n \norm{\varphi}_{\Cc^\alpha}  \,\diff \xi  \lesssim   { 1  \over \sqrt n}.$$
	
	Observe that $ \big| \Phi_{n,\xi} +\Phi_{n}^{\star} \big| \leq \Phi_{n,0} +\Phi_{n}^{\star} = \varphi$, see \eqref{eq:psi_xi+psi_T}. Then,  $\big|\oN_0  (\Phi_{n,\xi} +\Phi_{n}^{\star}  )\big| \leq\oN_0\varphi= 1$. Therefore, using \eqref{eq:case1-lambda}, we obtain
	$$\int_{\sqrt[6] n<|\xi|\leq \xi_0\sqrt n} {1\over |\xi|}\cdot \big| \Gamma_2(\xi) \big|  \,\diff \xi \lesssim \int_{-\xi_0\sqrt n}^{\xi_0\sqrt n} {1\over \sqrt[6] n} \cdot{1\over \sqrt n} e^{-{\varrho^2\xi^2\over 4}}    \,\diff \xi  \lesssim {  1\over \sqrt n}.$$
	\vskip 5pt
	
	\noindent\textbf{Case 2.} $|\xi|\leq \sqrt[6] n$.  In this case, by Lemma \ref{lemma:lambda-estimates}, we have
	\begin{equation} \label{eq:case2-lambda}
	\big|\lambda_{-{i\xi\over \sqrt n}}^n\big|\leq e^{-{\varrho^2\xi^2\over 3}} \quad \text{and}\quad \Big| e^{i\xi\sqrt n \gamma} \lambda_{-{i\xi\over \sqrt n}}^n-e ^{-{\varrho^2\xi^2\over 2}}  \Big|\lesssim {1\over \sqrt n}|\xi|^3e^{-{\varrho^2\xi^2\over 2}}.
	\end{equation}
	
	From \eqref{eq:psi_xi+psi_T}, it follows that $\norm{\Phi_{n,\xi} +\Phi_{n}^{\star}}_{\Cc^\alpha}$ is bounded by
	\begin{align*}
	\norm{\varphi}_{\Cc^\alpha} &+ \sum_{|k| \leq A\log n} \big|   e^{i \xi{k \over \sqrt n}}-1  \big|\cdot \norm{\chi_k\varphi}_{\Cc^\alpha}\lesssim \norm{\varphi}_{\Cc^\alpha} +  \sum_{|k| \leq A\log n} |\xi|{|k|\over \sqrt n} e^{\alpha |k|}\norm{\varphi}_{\Cc^\alpha} \\
	&\lesssim 1 +  {\sqrt[6] n (\log n)^2 n^{\alpha A}\over \sqrt n}  \lesssim  1  +  {(\log n)^2 \over \sqrt[6] n}  \lesssim 1, 
	\end{align*}
	where  we have used that $\norm{\chi_k\varphi}_{\Cc^\alpha} \lesssim e^{\alpha |k|} \norm{\varphi}_{\Cc^\alpha} \lesssim  n^{\alpha A}$ from Lemma \ref{lemma:norm-Phi} and the assumption that $\alpha A\leq 1/6$. It follows from the analyticity of $\xi\mapsto\oN_{i\xi}$ that
	$$  \Big\|\big(\oN_{-{i\xi\over \sqrt n}} -\oN_0\big)(\Phi_{n,\xi} +\Phi_{n}^{\star}  )  \Big\|_\infty  \lesssim {|\xi|\over \sqrt n}.$$
	We conclude, using \eqref{eq:case2-lambda}, that 
	$$\int_{|\xi|\leq\sqrt[6] n} {1\over |\xi|}\cdot \big| \Gamma_1(\xi) \big|  \,\diff \xi \lesssim \int_{|\xi|\leq\sqrt[6] n} {1\over |\xi|} \cdot  e^{-{\varrho^2\xi^2\over 3}} {|\xi|\over \sqrt n}   \,\diff \xi \lesssim { 1 \over \sqrt n}. $$
	
	For $\Gamma_2$, using that $\big|\oN_0  (\Phi_{n,\xi} +\Phi_{n}^{\star}  )\big| \leq 1$ as before together with  \eqref{eq:case2-lambda}, gives
	$$\int_{|\xi|\leq\sqrt[6] n} {1\over |\xi|}\cdot \big| \Gamma_2(\xi) \big|  \,\diff \xi \lesssim \int_{|\xi|\leq\sqrt[6] n} {1\over |\xi|}\cdot   {1\over \sqrt n}|\xi|^3e^{-{\varrho^2\xi^2\over 2}}  \,\diff \xi\lesssim {1 \over \sqrt n}.$$
	
	Together with Case 1, we deduce that 
	$$\Big|\int_{-\xi_0\sqrt n}^{\xi_0\sqrt n} {\Lambda_1(\xi) \over \xi}    \,\diff \xi\Big|  \lesssim {1 \over \sqrt n}.$$
	 The lemma follows.
\end{proof}

\begin{lemma} \label{lemma:theta-2-bound}
	We have
	$$\sup_{|b|\leq  \sqrt n}    \Big|  \int_{0}^{\xi_0\sqrt n} {\Theta_b^{(2)}(\xi)-\Theta_b^{(2)}(-\xi)  \over \xi}   \,\diff \xi  \Big| \lesssim {1 \over \sqrt n}.$$
\end{lemma}

\begin{proof}
	Recall that  $\widetilde h_n(\xi)= d_{n,x} e^{-{\varrho^2 \xi^2\over 2}} \oN_0 (\Phi_{n,\xi} +  \Phi_{n}^{\star})$ and $\widehat h(\xi) =  e^{-{\varrho^2 \xi^2\over 2}}$. Put 
	$$\Omega_1(\xi):=  e^{ib \xi} e^{-{\varrho^2 \xi^2\over 2}}\big(  \oN_0 (\Phi_{n,\xi} +  \Phi_{n}^{\star}) - 1 \big) \widehat{\vartheta_{\delta_n}}(\xi) $$
	and
	$$\Omega_2(\xi):=e^{ib \xi} e^{-{\varrho^2 \xi^2\over 2}} (d_{n,x}-1)   \oN_0 (\Phi_{n,\xi} +  \Phi_{n}^{\star}) \, \widehat{\vartheta_{\delta_n}}(\xi), $$
	so that $\Theta_b^{(2)}=\Omega_1+\Omega_2$.
	\vskip 5pt
	
	We estimate $\Omega_1$ first.
	Since $\chi_k$ is bounded by $1$ and it is supported by  $\Tc^u_k \subset \{|u| \geq |k| - 1\}$ (see Lemma \ref{lemma:partition-of-unity}), we have
	$$\oN_0 (\chi_k\varphi) = \int_{\P^{d-1}} \chi_k\varphi \, \diff \nu \leq \nu \{|u| \geq |k| - 1\} \norm{\varphi}_\infty  \lesssim e^{-\eta_* |k|},$$
	where in the last step we have used Property (1) of Definition \ref{def:admissible}.
	
	Using the identity \eqref{eq:psi_xi+psi_T}, the fact that $\oN_0 \mathbf \varphi =  1$ and  $\norm{\widehat{\vartheta_{\delta_n}}}_{\Cc^1}\lesssim 1$,  we get
	\begin{align*}
	\big|\Omega_1(\xi)\big| =\Big| e^{ib \xi} e^{-{\varrho^2 \xi^2\over 2}} \big(  \oN_0 (\Phi_{n,\xi} +  \Phi_{n}^{\star}) - 1 \big) \widehat{\vartheta_{\delta_n}}(\xi) \Big|=  e^{-{\varrho^2 \xi^2\over 2}} \big|  \oN_0 (\Phi_{n,\xi} +  \Phi_{n}^{\star}) - \oN_0\varphi    \big|\cdot \big|\widehat{\vartheta_{\delta_n}}(\xi)\big|\\
	\lesssim e^{-{\varrho^2 \xi^2\over 2}} \sum_{|k| \leq A\log n} \big|  e^{i \xi{k \over \sqrt n}}-1 \big|  \oN_0 ( \chi_k \varphi )   \lesssim e^{-{\varrho^2 \xi^2\over 2}}   \sum_{k \geq 0} | \xi| {|k| \over \sqrt n} e^{-\eta_* |k|} \lesssim e^{-{\varrho^2 \xi^2\over 2}}{|\xi|\over \sqrt n}.
	\end{align*}
	Therefore, the integral involving $\Omega_1$ is bounded by
	$$   \int_{-\xi_0\sqrt n}^{\xi_0\sqrt n} \Big|{\Omega_1(\xi)\over \xi} \Big|  \,\diff \xi   \lesssim  \int_{-\xi_0\sqrt n}^{\xi_0\sqrt n}  {1\over |\xi|} \cdot e^{-{\varrho^2 \xi^2\over 2}} {|\xi|\over \sqrt n} \norm{\varphi}_\infty \,\diff \xi \lesssim {1 \over \sqrt n}.$$
	
	\vskip 3pt
	
	It remains to estimate $\Omega_2$. By a direct computation using the identity \eqref{eq:psi_xi+psi_T} and that $ \big|\oN_0 (\Phi_{n,\xi} +  \Phi_{n}^{\star})\big| \leq 1$ (see the proof of Lemma \ref{lemma:theta-1-bound}), $|d_{n,x} -1| \lesssim \beta^n$, $|\widehat{\vartheta_{\delta_n}}|\leq 1$, and $|b|\leq \sqrt n$, gives that, for $|\xi|\leq \xi_0\sqrt n$,
	\begin{align*}
	\big| \Omega_2'(\xi)\big| &\lesssim |d_{n,x}-1| \Big( \big|b \oN_0  (\Phi_{n,\xi} +\Phi_{n}^{\star}  )(x)\big|+    \big|\sqrt n  \oN_0  (\Phi_{n,\xi} +\Phi_{n}^{\star}  )(x)\big| \\ 
	& \quad\quad \quad\quad \quad\quad \quad + \sum_{|k| \leq A\log n} \Big|{k\over \sqrt n} e^{i \xi \frac{k}{\sqrt n}}\oN_0 (\chi_k\varphi) (x)    \Big|  + \big|\oN_0  (\Phi_{n,\xi} +\Phi_{n}^{\star}  )(x)\big|\cdot \norm{\widehat{\vartheta_{\delta_n}}}_{\Cc^1} \Big)\\
	&\lesssim \beta^n \Big(  \sqrt n+\sqrt n + \sum_{|k| \leq A\log n} \oN_0(\chi_k\varphi) + 1  \Big)\lesssim \sqrt n\beta^n.
	\end{align*}
	Thus, by the mean value theorem, 
	$$\Big| \int_{0}^{\xi_0\sqrt n} {\Omega_2(\xi)-\Omega_2(-\xi)  \over \xi}   \,\diff \xi \Big| \leq 2\xi_0 \sqrt n \sup_{|\xi|\leq \xi_0\sqrt n} \big|\Omega_2'(\xi)\big|\lesssim n\beta^n \lesssim { 1 \over \sqrt n}.$$ 
	The lemma follows.
\end{proof}

We can now finish the proof of  Proposition \ref{prop:BE-coeff-1-varphi}.

\begin{proof}[End of the proof  of Proposition \ref{prop:BE-coeff-1-varphi}] 
	Recall that our goal is to prove \eqref{goal-1-varphi}. Estimate \eqref{eq:BE-main-estimate} together with Lemmas \ref{lemma:theta-1-bound} and \ref{lemma:theta-2-bound} give that
	$$\big|F_n(b)-H(b)\big| \lesssim { 1 \over  \sqrt n} \quad \text{for all} \quad b \in \R.$$
	Recall that $H(b): =\frac{1}{\sqrt{2 \pi}  \,\varrho} \int_{-\infty}^b e^{-\frac{s^2}{2 \varrho^2}} \, \diff s$. Coupling the last estimate with Lemma \ref{lemma:Ln-Fn} and the fact that $\sup_{b \in \R} \big|H(b)-H(b \pm 1/ \sqrt n)\big| \lesssim 1/ \sqrt n$, we get $ \big|\oL_n^b(\varphi)- H(b)\big| \lesssim 1 /  \sqrt n$, thus showing that  \eqref{goal-1-varphi} holds. Observe that all of our estimates are uniform in $x \in \P^{d-1}$ and $y\in (\P^{d-1})^*$. The proof of the proposition is complete.
\end{proof}

\subsection{The case of general $\psi \in \Hc$} In this subsection, we use Proposition \ref{prop:BE-coeff-1-varphi} to prove the general case of Theorem \ref{thm:BE-general}. Recall that our goal is to show \eqref{eq:E_n-R}.   The following lemma will allow us later on to consider only the case when $|b|\leq  \sqrt n$.

\begin{lemma}\label{b-leq-n-2}
	There exists a constant $C>0$ such that for all $n\geq 1$,
	$$\oE_n\big(\mathbf 1_{|t|\geq \sqrt n}\big)\leq C/ \sqrt n.  $$ 
\end{lemma}

\proof
We can assume that $n$ is large enough. Let $A>0$ be as before.
By \eqref{A-large-ineq}, $|u(S_n x)| \geq A\log n $ with probability $\lesssim 1 / \sqrt n$. Therefore,  by the definition of $\oE_n$, see \eqref{eq:En-def},   we have
$$\oE_n\big(\mathbf 1_{|t|\geq \sqrt n}\big)\leq \mathbf E\Big(\mathbf 1_{\left|{ \sigma(S_n,x) + u(S_n x)-n\gamma\over \sqrt n}\right| \geq \sqrt n }    \mathbf 1_{|u(S_n x)| < A\log n}\Big)+ C' /\sqrt n$$
for some constant $C'>0$.

Using that  $A\log n \leq n/2$ for $n$ large enough, the last quantity is bounded by  
$$\mathbf E\Big(\mathbf 1_{| \sigma(S_n,x) -n\gamma| \geq  n / 2}\Big)+ C' /\sqrt n = \mathbf P\Big(| \sigma(S_n,x) -n\gamma| \geq  n / 2\Big)+ C'/\sqrt n.$$
The result follows by applying Proposition \ref{prop:BQLDT} with $\ep=1/2$.
\endproof

We continue the proof of Theorem \ref{thm:BE-general}.  Recall from the introduction that $\Hc$ is the space of bounded Lipschitz functions $\psi$ on $\R$ such that  $\psi'$ is bounded and belongs to $L^1(\R)$, equipped with the norm
$$\norm{\psi}_\Hc:=\norm{\psi}_\infty+\norm{\psi'}_\infty +\norm{\psi'}_{L^1}.$$

We can assume $\norm{\psi}_\Hc \leq 1$. So, $\psi$ is a Lipschitz function on $\R$ such that $\norm{\psi}_\infty\leq 1$, $\norm{\psi'}_\infty\leq 1$ and $\norm{\psi'}_{L^1}\leq 1$. We can also assume that $\varphi$ is non-negative and $\norm{\varphi}_{\Cc^\alpha}\leq 1$.
Recall that  $J=(-\infty,b]$.  By Proposition \ref{prop:BE-coeff-1-varphi}, Theorem  \ref{thm:BE-general} holds when $\psi$ is a constant function. Hence,  we can subtract $\psi(b)$ from $\psi$ and assume that $\psi(b)=0$. In this case, $\mathbf 1_{(-\infty,b]}\psi$ is a  Lipschitz function on $\R$ with Lipschitz norm bounded by $1$.    By considering differences of such functions, we can  assume that $\psi$ is non-negative.
\vskip 5pt

The following regularization procedure will be used.
For every $n\geq 1$, consider a smooth cut-off function $\tau_n$ on $\R$ such that $$\tau_n=0  \,\,\text{ on }\,\, (-\infty,-\sqrt n-2], \quad \tau_n=1  \,\,\text{ on }\,\, [-\sqrt n,\infty)   \quad \text{and} \quad |\tau_n'|\leq 1 \,\,\text{ on }\,\, \R.$$

 \begin{lemma} \label{lemma:psi-tau-n}
 Assume that \eqref{eq:E_n-R} holds for $\psi\tau_n$ and $\varphi$ instead of $\psi$ and $\varphi$. Then it holds for  $\psi$ and $\varphi$. Moreover, \eqref{eq:E_n-R} holds for  $\psi$ and $\varphi$ when $|b| \geq \sqrt n$.
 \end{lemma}
 
 \begin{proof}
 We have that
  \begin{align*}
  \big|\oE_n ( & \mathbf 1_{(-\infty,b]} \psi \cdot \varphi) - \oR( \mathbf 1_{(-\infty,b]}  \psi \cdot \varphi) - \oE_n ( \mathbf 1_{(-\infty,b]} (\psi\tau_n)  \cdot \varphi)  + \oR( \mathbf 1_{(-\infty,b]}  (\psi\tau_n)  \cdot \varphi) \big| \\ &\leq \oE_n(|\Phi_n|) + \oR(|\Phi_n|),
 \end{align*}
 where $\Phi_n(t,w):= \mathbf 1_{(-\infty,b]} (\psi\tau_n) (t) \cdot \varphi(w) - \mathbf 1_{(-\infty,b]} \psi (t) \cdot \varphi(w)$. By the definition of $\tau_n$, the function $\Phi_n$ is bounded by a constant and vanishes when $t \geq - \sqrt n$. It follows from Lemma \ref{b-leq-n-2} that $\oE_n(|\Phi_n|) \leq  C  / \sqrt n$. The same estimate clearly holds for $\oR(|\Phi_n|)$. The first statement follows.
 
 For the second statement, we observe that if $b < - \sqrt n$, we have that $$ 0 \leq \mathbf 1_{(-\infty, -\sqrt n]} \psi\cdot \varphi - \mathbf 1_{(-\infty,b]} \psi\cdot \varphi    =  \mathbf 1_{(b,-\sqrt n]}   \psi \cdot \varphi  \leq \mathbf 1_{t  \leq - \sqrt n},$$ so $\big| \oE_n (\mathbf 1_{(-\infty,- \sqrt n]} \psi \cdot \varphi) - \oE_n (\mathbf 1_{(-\infty,b]} \psi \cdot \varphi)  \big| \lesssim 1 / \sqrt n$ by Lemma \ref{b-leq-n-2} and it is clear that the same estimate holds for $\oR$ instead of $\oE_n$. Therefore, \eqref{eq:E_n-R} holds in this case. When $b > \sqrt n$, we can repeat the above arguments with $1 - \oE_n$ (resp.  $1 - \oR$) instead  of $\oE_n$ (resp.  $\oR$) and obtain the same estimates.
   \end{proof}

 By the above lemma, we can assume from now on that $|b|\leq \sqrt n$ and, to prove Theorem \ref{thm:BE-general},  it will be enough to obtain \eqref{eq:E_n-R}  for $(\psi\tau_n)(t)\cdot \varphi(w)$ instead of $\psi(t)\cdot \varphi(w)$. In order to obtain the desired result for $(\psi\tau_n)  \cdot \varphi$, the following estimates will be needed. 

\begin{lemma}\label{lemma-tau}
	Let $\tau_n$ be as above and $b \in \R$ be  such that $|b|\leq \sqrt n$. Let $\psi$ be a  function in $\Hc$ such that $\norm{\psi}_\Hc\leq 1$ and $\psi(b)=0$. Then, we have
	$$ \big|\widehat{(\mathbf 1_{(-\infty,b]}\psi\tau_n)}(\xi)\big|\leq 2\sqrt n +2         \quad
	\text{and}
	\quad \big|\widehat{(\mathbf 1_{(-\infty,b]}\psi\tau_n)}(\xi)\big|\leq 5/|\xi| \quad\text{for}\quad \xi\neq 0.$$
\end{lemma}

\begin{proof}
	The assumption $\norm{\psi}_\Hc\leq 1$ implies that $\norm{\psi}_\infty\leq 1$, $\norm{\psi'}_\infty\leq 1$ and $\norm{\psi'}_{L^1}\leq 1$. By the definition of $\tau_n$, we have  $|(\psi\tau_n)'|\leq |\psi'|$ on $\R\setminus (-\sqrt n-2, -\sqrt n)$.	Also,  $|(\psi\tau_n)'(t)|\leq 2$ on $[-\sqrt n-2, -\sqrt n]$, since $|\tau_n'|\leq 1$. Therefore,
	\begin{equation}\label{derivative-int-1}
	\int_{-\infty}^\infty  \big|(\psi\tau_n)' (t)\big|\,\diff t \leq \int_{-\infty}^\infty \big|\psi'(t)\big|\,\diff t +\int_{-\sqrt n-2}^{-\sqrt n} 2 \,\diff t \leq 1 +4 =5.
	\end{equation}
	
	By the definition of Fourier transform, 
	$$\big|\widehat{(\mathbf 1_{(-\infty,b]}\psi\tau_n)}(\xi)\big| \leq  \| \mathbf 1_{(-\infty,b]}\psi\tau_n  \|_{L^1}\leq \int_{- \sqrt n - 2}^b \diff t \leq 2\sqrt n +2$$
	since $|b|\leq \sqrt n$. This gives the first inequality.
	
	For the second inequality,  using integration by parts, $\widehat{(\mathbf 1_{(-\infty,b]}\psi\tau_n)}(\xi)$ is equal to 
	$$ \int_{-\infty}^b  (\psi\tau_n)(t)e^{-it\xi}   \,\diff t  = \Big[  (\psi\tau_n)(t)\cdot {e^{-it\xi}\over -i\xi} \Big]_{-\infty}^b- \int_{-\infty}^b  (\psi\tau_n)'(t) \cdot {e^{-it\xi}\over -i\xi} \,\diff t   .  $$
	Notice that $ (\psi\tau_n)(-\infty)= (\psi\tau_n)(b)=0$,  where we have used the assumption that $\psi(b)=0$.  Therefore,  the first term on the right hand side vanishes when $\xi\neq 0$. Using \eqref{derivative-int-1}, we conclude that
	\begin{equation*}
	\big|\widehat{(\mathbf 1_{(-\infty,b]}\psi\tau_n)}(\xi)\big|\leq 5/|\xi| \quad\text{for}\quad \xi\neq 0,
	\end{equation*}
	yielding the second estimate.
\end{proof}

Let  $\vartheta_\delta$ be the function from Lemma \ref{lemma:vartheta} and take $\delta_n:=(\xi_0 \sqrt n)^{-1/2}$, where $\xi_0$ is a small constant satisfying Lemma \ref{lemma:lambda-estimates}. 
Define $$\Psi_n:=(\mathbf 1_{(-\infty,b]}\psi\tau_n) * \vartheta_{\delta_n}.$$

Recall that, after Lemma \ref{lemma:psi-tau-n}, our goal is to obtain \eqref{eq:E_n-R} for  $(\psi\tau_n) \cdot \varphi$ instead of $\psi \cdot \varphi$. From Lemma \ref{lemma:convolution-lip}, we have that $\norm{\Psi_n - \mathbf 1_{(-\infty,b]}\psi\tau_n}_\infty \lesssim 1 / \sqrt n$. Combined with the estimate \eqref{A-large-ineq}, we see that it is enough to show that
\begin{equation}\label{goal-psi-varphi}
\big| \oS_n  -  \oR_n \big| \lesssim 1/\sqrt n,
\end{equation}
where
$$ \oS_n :=\mathbf E\Big( \Psi_n\Big({\sigma(S_n,x) + u(S_n x) - n \gamma \over \sqrt n}\Big)   \mathbf 1_{|u(S_n x)| < A\log n} \,  \varphi(S_n x) \Big)$$
and
$$\oR_n:=  \oR\big(  \Psi_n(t)\cdot \varphi(w)\big).$$

Let $\chi_k$ be the functions defined in Lemma \ref{lemma:partition-of-unity}  and $\Phi_{n}^{\star}$ be the function defined in \eqref{eq:Phi-star-def}. Set
$$\oA_n:= \sum_{|k| \leq A\log n} \mathbf E\Big(  \Psi_n\Big({\sigma(S_n,x) - n \gamma  - k\over \sqrt n}\Big) (\chi_k \varphi)(S_n x) \Big) + \mathbf E\Big(  \Psi_n\Big({\sigma(S_n,x) - n \gamma  \over \sqrt n}\Big) \Phi_{n}^{\star}(S_n x) \Big).$$

\begin{lemma}\label{A-S-lemma}
	We have $|\oA_n-\oS_n|\lesssim 1/\sqrt n$.
\end{lemma}

\begin{proof}

Recall that $ \sum_{|k| \leq A\log n-1}(\chi_k\varphi)(w)\leq \mathbf 1_{  |u(w)| \leq  A \log n  } \cdot \varphi(w) \leq  \sum_{|k| \leq A\log n+1}(\chi_k\varphi)(w)$, so $\oS_n^- \leq \oS_n \leq \oS^+$, where
$$\oS_n^\pm := \sum_{|k| \leq A\log n  \pm 1} \mathbf E\Big(   \Psi_n\Big({\sigma(S_n,x) +u(S_n x) - n \gamma \over \sqrt n}\Big)  \,(\chi_k\varphi)(S_n x) \Big).$$
	
	Since $\mathbf 1_{(-\infty, b]}\psi\tau_n$ is Lipschitz with a Lipschitz norm bounded uniformly in $n$, we see that $\Psi_n'$ is bounded by a constant independent of $n$. Therefore, 
	$$\Big|\Psi_n\Big({\sigma(S_n,x) - n \gamma  - k\over \sqrt n}\Big) -\Psi_n\Big({\sigma(S_n,x) + u(S_n x) - n \gamma \over \sqrt n}\Big) \Big|\lesssim \Big|{u(S_n x)+k\over \sqrt n } \Big|,$$
	which is bounded by $1/\sqrt n$ when $S_n x\in \supp(\chi_k)$. Using that $\Psi_n$ is uniformly bounded, we have
\begin{align*}
|\oS_n^+ - \oA_n| \lesssim \sum_{|k| \leq A\log n} \mathbf E\Big({1\over \sqrt n}  (\chi_k\varphi)(S_n x)   \Big) + \E \big(\chi_{k_0}\varphi)(S_n x)\big)   + \,  \E \big(\chi_{-k_0}\varphi)(S_n x)\big)    + \E\big(\Phi_{n}^{\star}(S_n x)\big),
\end{align*}
where $k_0:=\lfloor A\log n\rfloor+1$.	Since $\sum \chi_k\varphi \leq \norm{\varphi}_\infty \leq 1$, the first term on the right hand side is bounded by $1/\sqrt n$. For the remaining terms,  observe that  $\chi_{\pm k_0}$ and  $\Phi_{n}^{\star}  $ are supported by $\big\{ |u| \geq A\log n - 1\big\}$, see Lemma \ref{lemma:norm-Phi}, so the last three terms above are bounded by a constant times $\mathbf P\big( |u(S_n x)| \geq A\log n - 1 \big)$. By the same argument as the ones showing \eqref{A-large-ineq}, the last quantity is $\lesssim 1/\sqrt n$. It follows that $|\oS_n^+ - \oA_n| \lesssim 1/\sqrt n$. An analogous argument shows that $|\oA_n - \oS_n^-| \lesssim 1/\sqrt n$. Since  $\oS_n^- \leq \oS_n \leq \oS^+$, the lemma follows.
\end{proof}

We now resume the proof of Theorem \ref{thm:BE-general}. Clearly, $\Psi_n$ is integrable and $\widehat{\Psi_n}$ is supported by $[-\delta_n^{-2},\delta_n^{-2}]$. Using the inverse Fourier transform for $\Psi_n$, Fubini's theorem and \eqref{eq:markov-op-iterate}, we have
\begin{align*}
\mathbf E \Big(  \Psi_n\Big( {\sigma(S_n, x) - n \gamma-k \over \sqrt n}\Big)& (\chi_k\varphi)(S_n x) \Big)=\int_{G} \Psi_n\Big( {\sigma(g, x) - n \gamma-k \over \sqrt n}\Big)  (\chi_k\varphi)(g x) \,\diff \mu^{*n}(g) \\
&={1\over 2\pi}\int_G\int_{-\infty}^\infty \widehat {\Psi_n}(\xi) e^{i\xi {\sigma(g,x)-n\gamma -k \over \sqrt n} }  (\chi_k\varphi)(g x) \,\diff \mu^{*n}(g) \diff \xi \\
&={1\over 2\pi}\int_{-\infty}^\infty \widehat{\Psi_n}(\xi)\cdot e^{-i\xi\sqrt n \gamma } e^{-i\xi{k\over \sqrt n}} \oP_{{i\xi\over \sqrt n}}^n (\chi_k\varphi)(x)  \,\diff \xi. 
\end{align*}
Similarly,
$$\mathbf E\Big(  \Psi_n\Big({\sigma(S_n,x) - n \gamma  \over \sqrt n}\Big) \Phi_{n}^{\star}(S_n x) \Big)  = {1\over 2\pi}\int_{-\infty}^\infty \widehat{\Psi_n}(\xi)\cdot e^{-i\xi\sqrt n \gamma }  \oP_{{i\xi\over \sqrt n}}^n \Phi_{n}^{\star}(x)  \,\diff \xi .$$
Since the support of $\widehat{\Psi_n}$ is contained in $[-\delta_n^{-2},\delta_n^{-2}]$, it follows that
$$\oA_n={1\over 2\pi}\int_{-\xi_0\sqrt n}^{\xi_0 \sqrt n} \widehat{\Psi_n}(\xi)\cdot e^{-i\xi\sqrt n \gamma } \oP_{{i\xi\over \sqrt n}}^n (\Phi_{n,-\xi}+ \Phi_{n}^{\star} )(x)  \,\diff \xi,$$
where $\Phi_{n}^{\star}$ and $\Phi_{n,\xi}$ are defined in \eqref{eq:Phi-star-def} and  \eqref{eq:Phi-xi-def} respectively.

 Using that the  Fourier transform of $e^{-{\varrho^2 t^2\over 2}} $ is $\sqrt{2\pi} \varrho^{-1} e^{-\frac{\xi^2}{2 \varrho^2}}$  we have that
\begin{eqnarray*}
	\oR_n = \oR\big(\Psi_n(t)\cdot \varphi(w)\big) &=&   {1\over \sqrt{2 \pi} \varrho} \int_{\P^{d-1}} \int_{-\infty}^\infty e^{-\frac{s^2}{2 \varrho^2}} \Psi_n(s)  \cdot \varphi(w)\, \diff s\diff\nu(w) \\
	&=&  {1\over 2\pi}  (\oN_0 \varphi)\int_{-\infty}^\infty\int_{-\infty}^\infty e^{-i\xi s}e^{-{\varrho^2 \xi^2\over 2}} \Psi_n (s)  \,\diff \xi \diff s \\
	&=&{1\over 2\pi}  (\oN_0 \varphi)\int_{-\infty}^\infty e^{-{\varrho^2 \xi^2\over 2}} \widehat{\Psi_n} (\xi)  \,\diff \xi.
\end{eqnarray*}
It follows that
$$\oA_n - \oR_n={1\over 2\pi}\int_{-\xi_0\sqrt n}^{\xi_0 \sqrt n} \widehat{\Psi_n}(\xi)\Big[ e^{-i\xi\sqrt n \gamma } \oP_{{i\xi\over \sqrt n}}^n (\Phi_{n,-\xi}+ \Phi_{n}^{\star} )(x) -   e^{-{\varrho^2 \xi^2\over 2}}\oN_0\varphi   \Big] \,\diff \xi.$$

\begin{lemma}\label{A-R-lemma}
	We have  $\big|\oA_n - \oR_n \big|\lesssim 1/\sqrt n$.
\end{lemma}

\begin{proof}
	Define
	$$ \Lambda_1(\xi):=  e^{-i\xi\sqrt n \gamma} \lambda_{{i\xi\over \sqrt n}}^n \oN_{{i\xi\over \sqrt n}} (\Phi_{n,-\xi} +\Phi_{n}^{\star}  ) (x) -      e^{-{\varrho^2 \xi^2\over 2}}   \oN_0 (\Phi_{n,-\xi} +\Phi_{n}^{\star}  )  (x), $$ 
	$$\Lambda_2(\xi):= e^{-i\xi\sqrt n \gamma} \oQ_{{i\xi\over \sqrt n}}^n  (\Phi_{n,-\xi} +\Phi_{n}^{\star}  )(x) -e^{-i\xi\sqrt n \gamma}\oQ_0^n  (\Phi_{n,-\xi} +\Phi_{n}^{\star}  )(x),  $$
	$$ \Lambda_3(\xi):=e^{-i\xi\sqrt n \gamma}\oQ_0^n  (\Phi_{n,-\xi} +\Phi_{n}^{\star}  )(x)$$
	and
	$$\Omega_1(\xi):= e^{-{\varrho^2 \xi^2\over 2}}  \oN_0 (\Phi_{n,-\xi} +  \Phi_{n}^{\star}) - e^{-{\varrho^2 \xi^2\over 2}}\oN_0\varphi.$$
	Then,
	$$e^{-i\xi\sqrt n \gamma } \oP_{{i\xi\over \sqrt n}}^n (\Phi_{n,-\xi}+ \Phi_{n}^{\star} )(x) -   e^{-{\varrho^2 \xi^2\over 2}}\oN_0\varphi =\Lambda_1(\xi)+\Lambda_2(\xi)+\Lambda_3(\xi)+\Omega_1(\xi).$$
	
	From Lemma \ref{lemma-tau}, we see that $\big|\widehat{\Psi_n}(\xi)\big|\lesssim 1/ |\xi|$ for $\xi\neq 0$. Using this estimate, we can repeat the arguments in the proofs of Lemmas \ref{lemma:theta-1-bound}  and \ref{lemma:theta-2-bound} and deduce that 
	$$\int_{-\xi_0\sqrt n}^{\xi_0 \sqrt n}\Big| \widehat{\Psi_n}(\xi)\Lambda_1(\xi) \Big|\,\diff \xi,\quad \int_{-\xi_0\sqrt n}^{\xi_0 \sqrt n}\Big| \widehat{\Psi_n}(\xi)\Lambda_2(\xi) \Big|\,\diff \xi,\quad \int_{-\xi_0\sqrt n}^{\xi_0 \sqrt n}\Big| \widehat{\Psi_n}(\xi)\Omega_1(\xi) \Big|\,\diff \xi   $$
	are all $\lesssim 1/\sqrt n$. 
	
	It remains to estimate the integral involving $\Lambda_3$. 
	Recall from Lemma \ref{lemma-tau} that $\big|\widehat{\Psi_n}(\xi)\big|\leq 2\sqrt n +2$. Using $\big|\oQ_0^n  (\Phi_{n,-\xi} +\Phi_{n}^{\star}  )(x) \big|\lesssim \beta^n\norm{ \Phi_{n,-\xi} +\Phi_{n}^{\star}  }_{\Cc^\alpha}\lesssim \beta^n \sqrt[6] n$ for some $0<\beta<1$, see Lemma \ref{lemma:norm-Phi}, we get 
	$$\int_{-\xi_0\sqrt n}^{\xi_0 \sqrt n}\Big| \widehat{\Psi_n}(\xi)\Lambda_3(\xi) \Big|\,\diff \xi\lesssim \sqrt n (2\sqrt n+2)\beta^n \sqrt[6] n\lesssim {1\over \sqrt n}. $$
	This ends the proof of the lemma.
\end{proof}

\proof[End of the proof of  Theorem \ref{thm:BE-general}]
From the discussion after Lemma \ref{lemma-tau}, the theorem will follow if we show that \eqref{goal-psi-varphi} holds. This estimate is an immediate consequence of  Lemmas \ref{A-S-lemma} and \ref{A-R-lemma}. The proof of the theorem is complete.
\endproof

%%%%%%%%%%%%
\section{Local Limit Theorem} \label{sec:LLT}

In this section, we prove Theorem \ref{thm:LLT-general}. We fix $x \in \P^{d-1}$ from now on. For a function $\Phi$ on $\R\times\P^{d-1}$, define
\begin{equation*} 
\Ac_n(\Phi)(t) := \sqrt{n}  \,\mathbf E \Big( \Phi\Big( t+ \sigma(S_n,x)+u(S_nx) - n \gamma, S_n x\Big) \Big)
\end{equation*}
and
	$$\Gc_n(\Phi)(t) := \frac{1}{\sqrt{2 \pi} \varrho} e^{-\frac{t^2}{2 \varrho^2 n}}  \int_{\R\times\P^{d-1}} \Phi(s,w)\,\diff s\diff \nu(w).$$
	
Then, the conclusion of	Theorem \ref{thm:LLT-general} is equivalent to
\begin{equation} \label{eq:LLT-main-limit}
\lim_{n \to \infty} \sup_{t\in\R} \big|\Ac_n(\Phi)(t) - \Gc_n(\Phi)(t)  \big| = 0.
\end{equation}

\begin{proposition} \label{prop:LLT-product-test-fcn}
Theorem \ref{thm:LLT-general} holds when $\Phi(s,w) = \psi(s) \varphi(w)$ with $\psi$ continuous with compact support in $\R$ and $\varphi \in \Cc^\alpha(\P^{d-1})$.
\end{proposition}

As in Section \ref{sec:BE}, we can assume that  ${1\over 2}\leq \varphi\leq 2$,  $\oN_0 \varphi=1$ and $\norm{\varphi}_{\Cc^\alpha}\leq 2$, since the problem is linear on $\varphi$ and these functions span the space $\Cc^\alpha(\P^{d-1})$. We can also assume that  $\psi$ is non-negative and $\|\psi\|_\infty \leq 1$. We first consider the upper bound in the limit \eqref{eq:LLT-main-limit}. The lower bound is analogous and will be treated in the end of the proof.

 For $t \in \R$ and $k \in \Z$, let $t_k \in [-k-1,-k+1]$ be a point such that $$\sup_{s \in [-k-1,-k+1]}  \psi(t+ \sigma(S_n,x) + s - n \gamma) =  \psi(t+ \sigma(S_n,x) + t_k - n \gamma).$$
  We note that the point $t_k$ might not be unique and can depend on $t,n$ and $x$, but this won't be an issue in the proof. Therefore, we omit this dependence from our notation.
  
   For later use, consider the translates of $\psi$ given by $$\psi_{t,k}(s):=\psi(s+t+t_k).$$
 
 Observe that, since $\psi$ is compactly supported, for fixed $t,s\in\R$, we have that $\psi_{t,k}(s)\neq 0$ for only finitely many $k$'s.
 
 From Corollary \ref{cor:LDT-admissible}, there is a constant $A>0$ such that
 \begin{equation*}
\mu^{*n} \big\{g\in G:\,  |u(gx)| \geq  A \log n \big\} \lesssim 1 / n.
\end{equation*}

	In order to simplify the notation, for a function $\Psi$ on $\R \times \P^{d-1}$, let
	\begin{equation*} 
\Ec_n\big(\Psi\big):= \sqrt {n} \, \mathbf E\Big( \Psi\big( \sigma(S_n,x)-n\gamma ,S_n x \big)  \Big),
\end{equation*}
and, as in Section \ref{sec:BE}, consider the function on $\P^{d-1}$ given by
	\begin{equation*}
		\Phi_{n}^{\star}  (w):= \varphi(w) - \sum_{|k| \leq A\log n} (\chi_k\varphi)(w),
	\end{equation*}	
	where $\chi_k$ are the functions in Lemma \ref{lemma:partition-of-unity}.  We use the same notation as in Section \ref{sec:BE}.
	
	Denote $\Ac_n(t):= \Ac_n(\psi \cdot \varphi)(t)$ for simplicity and set
	$$\Bc_n(t):=  \sum_{|k| \leq A \log n} \Ec_n\big(\psi_{t,k}\cdot \chi_k \varphi  \big)+   \Ec_n\big(\psi_{t,0}\cdot \Phi_n^\star \big).$$
	
	\begin{lemma} \label{lemma:llt-ineq-1}
	There exists a constant $C_1>0$,  independent of $n$,  such that,  for all $t \in \R$,   $$\Ac_n(t) \leq \Bc_n(t) + C_1 / \sqrt n.$$
	\end{lemma}
	
\begin{proof}
Using that $\mathbf P \big(|u(S_n x)| \geq A \log n \big) \lesssim 1 / n$, we obtain
\begin{align*}
	\Ac_n(t) \leq \sqrt{n}  \,   \mathbf E \Big(  \psi(t+ \sigma(S_n,x) + u (S_n x) - n \gamma)  \mathbf 1_{|u(S_n x)| \leq A \log n} \varphi(S_n x) \Big)    + O \big( 1 / \sqrt n \big) .
	\end{align*}
	Observe that,  when $S_n x \in\supp(\chi_k)$, we have $-k-1 \leq u(S_n x) \leq -k+1$,  so in this case
	$$  \psi(t+ \sigma(S_n,x) + u(S_nx) - n \gamma) \leq \psi(t+ \sigma(S_n,x) + t_k - n \gamma) = \psi_{t,k}\big(\sigma(S_n,x)  - n \gamma\big),$$
	where we have used the definitions of $t_k$ and $\psi_{t,k}$. 	Using that $\mathbf 1_{|u|\leq A \log n} \leq \sum_{|k| \leq A \log n + 1} \chi_k$, ${1\over 2}\leq \varphi\leq 2$ and taking the expectation, it follows that
	\begin{align*}
	&\mathbf E \Big(  \psi(t+ \sigma(S_n,x) + u (S_n x) - n \gamma)  \mathbf 1_{|u(S_n x)| \leq A \log n} \varphi(S_n x) \Big) \\
	&\leq  \sum_{|k| \leq A \log n + 1} \E \Big( \psi_{t,k}\big(\sigma(S_n,x)-n\gamma\big)(\chi_k \varphi)(S_n x) \Big)    \\ &\leq  \sum_{|k| \leq A \log n} \E \Big(\psi_{t,k}\big(\sigma(S_n,x)-n\gamma\big) (\chi_k \varphi) (S_n x) \Big) +   2 \E\Big(\chi_{k_0}(S_n x) \Big) + 2\E\Big(\chi_{-k_0}(S_n x) \Big),
	\end{align*}
where $k_0:= \lfloor A \log n \rfloor + 1$.  

Since $\chi_{k_0}$ and $\chi_{-k_0}$ are both supported by $\{|u| > A \log n -1\}$ and bounded by one, we see that the last two terms above are bounded by $\mathbf P \big( |u(S_n x)| \geq  A \log n - 1 \big) \lesssim 1 / n$. Hence, there is a constant $C_1 >0$ such that
	\begin{equation*}
	\Ac_n(t) \leq	\sum_{|k| \leq A \log n} \Ec_n\big(\psi_{t,k}\cdot \chi_k \varphi  \big)+{C_1 \over \sqrt n} \leq \Bc_n(t)  +{C_1 \over \sqrt n},
	\end{equation*}
	proving the lemma.
\end{proof}

	By Lemma \ref{lemma:conv-fourier-approx}, for every $0<\delta \leq 1$, there exists a smooth function $\psi^+_{\delta}$  such that $\widehat {\psi^+_{\delta}}$ has support in $[-\delta^{-2},\delta^{-2}]$,  $$\psi\leq \psi^+_\delta,\quad \lim_{\delta\to 0} \psi^+_{\delta} =\psi    \quad \text{and} \quad  \lim_{\delta\to 0} \big \|\psi^+_{\delta} -\psi \big \|_{L^1} = 0.$$ 
	Moreover,  $\norm{\psi_{\delta}^+}_\infty$, $\norm{\psi_{\delta}^+}_{L^1}$ and $\|\widehat{\psi^+_{\delta}}\|_{\Cc^1}$ are bounded by a constant independent of $\delta$.
	
	As above, for $t\in\R$ and $k\in\N$, we consider the translations
	$$ \psi_{t,k}^+(s):=\psi_{\delta}^+(s+t+t_k). $$
	We omit the dependence on $\delta$ in order to ease the notation. Define also
	\begin{equation} \label{eq:R-def}
	\Bc_n^+(t):= \sum_{|k| \leq A \log n} \Ec_n\big(\psi_{t,k}^+\cdot \chi_k \varphi  \big)+  \Ec_n\big(\psi_{t,0}^+\cdot \Phi_n^\star  \big).	
\end{equation}

	Clearly, we have $\Bc_n(t)\leq \Bc_n^+(t)$. From the definition of $\Ec_n$, Fourier inversion formula and Fubini's theorem, we have
	\begin{align*}
	\Ec_n\big(\psi_{t,k}^+\cdot \chi_k \varphi \big)&=\sqrt n \, \int_{G} \psi_{\delta}^+\big(\sigma(g,x)-n\gamma+t+t_k\big) \cdot (\chi_k \varphi)(gx) \,\diff \mu^{*n}(g)\\
	&={\sqrt n\over 2\pi}\int_{G} \int_{-\infty}^\infty \widehat{\psi_{\delta}^+}(\xi) e^{i\xi(\sigma(g,x)-n\gamma+t+t_k )} \cdot (\chi_k \varphi)(gx) \,\diff \xi\diff\mu^{*n}(g)\\
	&={\sqrt n\over 2\pi}\int_{-\infty}^\infty  \widehat{\psi_{\delta}^+}(\xi) e^{i\xi(t+t_k)}\cdot e^{-i\xi n\gamma}\oP^n_{i\xi}(\chi_k \varphi)(x) \,\diff \xi,
	\end{align*}
	where in the last step we have used \eqref{eq:markov-op-iterate}.
	
	Recall that $\supp\big( \widehat{\psi_{\delta}^+} \big)\subset [-\delta^{-2},\delta^{-2}]$. So, after the change of variables $\xi \mapsto \xi / \sqrt n$, the above identity becomes 
\begin{equation} \label{eq:E_n-psi+}
	\Ec_n\big(\psi_{t,k}^+\cdot\chi_k \varphi \big) ={1\over 2\pi}\int_{-\delta^{-2} \sqrt n}^{\delta^{-2} \sqrt n} \widehat{\psi_{\delta}^+}\Big({\xi\over \sqrt n}\Big) e^{i\xi{t+t_k\over \sqrt n}}\cdot e^{-i\xi \sqrt n\gamma}\oP^n_{{i\xi\over \sqrt n}}(\chi_k \varphi)(x) \,\diff \xi.
\end{equation}

	A similar computation yields $$\Ec_n\big(\psi_{t,0}^+ \cdot \Phi_n^\star \big) = {1\over 2\pi}\int_{-\delta^{-2} \sqrt n}^{\delta^{-2} \sqrt n} \widehat{\psi_{\delta}^+}\Big({\xi\over \sqrt n}\Big) e^{i\xi{t \over \sqrt n}}\cdot e^{-i\xi \sqrt n\gamma}\oP^n_{{i\xi\over \sqrt n}}\Phi_{n}^{\star} (x) \,\diff \xi.$$

	Define
	\begin{equation} \label{eq:Phi-xi-def-2}
	\Phi_{n,\xi} (w):= \sum_{|k| \leq A \log n}  e^{i \xi{t_k \over \sqrt n}}(\chi_k\varphi)(w).
	\end{equation}
	We use again the same notation as in Section \ref{sec:BE} to denote a slightly different function. We use here the factor $-t_k$ instead of $k$. Recall that $t_k \in [-k-1,-k+1]$. Using this notation and the above computations, \eqref{eq:R-def} becomes
\begin{equation} \label{eq:R-formula} 
\Bc_n^+(t)= {1\over 2\pi}\int_{-\delta^{-2} \sqrt n}^{\delta^{-2} \sqrt n} \widehat{\psi_{\delta}^+}\Big({\xi\over \sqrt n}\Big) e^{i\xi{t\over \sqrt n}}\cdot e^{-i \xi \sqrt n\gamma}\oP_{{i\xi\over \sqrt n}}^n(\Phi_{n,\xi} +\Phi_{n}^{\star}  )(x) \,\diff \xi.
\end{equation}

Fix $\alpha>0$ such that $$\alpha A\leq 1/6 \quad \text{ and } \quad \alpha \leq \alpha_0,$$ where $0<\alpha_0<1$ is the exponent appearing in Theorem \ref{thm:spectral-gap}. Then, all the results of Subsection \ref{subsec:markov-op} apply to the operators $\xi\mapsto\oP_{i\xi}$ acting on $\Cc^\alpha(\P^{d-1})$.

The following result is the analogue of Lemma \ref{lemma:norm-Phi}. The fact that $|t_k| \leq |k| + 1$ can be used to show that the same estimates are still valid in the present case.  Recall that we are assuming that $\norm{\varphi}_{\Cc^\alpha} \leq 2$.

\begin{lemma}  \label{lemma:norm-Phi-2}
	Let $\Phi_{n,\xi}, \Phi_{n}^{\star}$ be the functions on $\P^{d-1}$ defined above.  Then,  the following identity holds
	\begin{equation} \label{eq:psi_xi+psi_T-2}
	\Phi_{n,\xi} + \Phi_{n}^{\star} = \mathbf \varphi + \sum_{|k| \leq A \log n} \big(  e^{i \xi{t_k \over \sqrt n}} - 1 \big) \chi_k\varphi.
	\end{equation}
	Moreover,  for every exponent $0<\alpha\leq \alpha_*$, we have $\norm{\chi_k\varphi}_{\Cc^\alpha}\lesssim e^{\alpha |k|}$ and there  is a constant $C>0$ independent of $\xi$ and $n$ such that   
	\begin{equation}  \label{eq:norm-Phi-2}
	\norm{\Phi_{n,\xi} }_{\Cc^\alpha}\leq C \,  n^{\alpha A} \quad\text{and}\quad  \norm{\Phi_{n}^{\star}  }_{\Cc^\alpha} \leq C \,  n^{\alpha A}.
	\end{equation}
	In addition,  $\Phi_{n}^{\star}  $ is supported by $\big\{|u| \geq A \log n -1\big\}$.
\end{lemma}

 	Define  
 	\begin{equation} \label{eq:S-def} 
 	\Gc_n^+(t):={1\over 2\pi}\widehat{\psi_{\delta}^+}(0)  \int_{-\infty}^\infty e^{i\xi{t\over \sqrt n}} e^{-{ \varrho^2\xi^2 \over 2}}  \,\diff \xi ={1\over \sqrt{2\pi} \, \varrho} e^{-{t^2\over2\varrho^2 n}}\int_{\R} \psi_{\delta}^+ (s)\,\diff s,
 	\end{equation}
	where in the second equality we have used the fact that the inverse Fourier transform of $e^{-{ \varrho^2\xi^2 \over 2}}$ is $ {1\over \sqrt{2\pi} \, \varrho}  e^{-{t^2\over2\varrho^2}}$.

   \begin{lemma}\label{lemma-R-S}
	Fix $0< \delta \leq 1$. Then, there exists a constant $C_{\delta}>0$ such that, for all $n \geq 1$,
	$$\sup_{t\in\R}\big|\Bc_n^+(t)-\Gc_n^+(t) \big| \leq {C_{\delta}\over \sqrt[3] n}.$$
	\end{lemma}
	\begin{proof}
	Let $\xi_0>0$ be a small constant, so that Lemma \ref{lemma:lambda-estimates} applies. In particular, the  decomposition of $\oP_z$ in Proposition \ref{prop:spectral-decomp} holds for $|z| \leq \xi_0$. Using that decomposition, \eqref{eq:R-formula} and \eqref{eq:S-def}, we can write $$\Bc_n^+(t)-\Gc_n^+(t) = \Lambda_n^1(t) + \Lambda_n^2(t)  + \Lambda_n^3(t)  + \Lambda_n^4(t)  + \Lambda_n^5(t),$$ where
	$$\Lambda_n^1(t):={1\over 2\pi}\int_{-\xi_0 \sqrt n}^{\xi_0 \sqrt n} e^{i\xi{t\over \sqrt n}} \Big[ \widehat{\psi_{\delta}^+}\Big({\xi\over \sqrt n}\Big)e^{-i \xi \sqrt n\gamma} \lambda_{{i\xi\over \sqrt n}}^n\oN_0(\Phi_{n,\xi} +\Phi_{n}^{\star}  )-\widehat{\psi_{\delta}^+}(0)    e^{-{ \varrho^2\xi^2 \over 2}}\Big]  \,\diff \xi ,$$
	$$\Lambda_n^2(t):= {1\over 2\pi}\int_{-\xi_0 \sqrt n}^{\xi_0 \sqrt n} e^{i\xi{t\over \sqrt n}}\Big[ \widehat{\psi_{\delta}^+}\Big({\xi\over \sqrt n}\Big) e^{-i \xi \sqrt n\gamma} \lambda_{{i\xi\over \sqrt n}}^n\big(\oN_{{i\xi\over \sqrt n}}-\oN_0\big)(\Phi_{n,\xi} +\Phi_{n}^{\star}  ) (x) \Big]  \,\diff \xi , $$
	$$\Lambda_n^3(t):=  {1\over 2\pi}\int_{-\xi_0 \sqrt n}^{\xi_0 \sqrt n} e^{i\xi{t\over \sqrt n}}  \widehat{\psi_{\delta}^+}\Big({\xi\over \sqrt n}\Big)  e^{-i \xi \sqrt n\gamma} \oQ_{{i\xi\over \sqrt n}}^n(\Phi_{n,\xi} +\Phi_{n}^{\star}  )(x)   \,\diff \xi,   $$
	$$ \Lambda_n^4(t):=  {1\over 2\pi}\int_{\xi_0\sqrt n \leq|\xi|\leq\delta^{-2} \sqrt n}e^{i\xi{t\over \sqrt n}} \widehat{\psi_{\delta}^+}\Big({\xi\over \sqrt n}\Big) e^{-i \xi \sqrt n\gamma} \oP_{{i\xi\over \sqrt n}}^n(\Phi_{n,\xi} +\Phi_{n}^{\star}  )(x) \,\diff \xi   $$
	and 
	$$ \Lambda_n^5(t):=  - {1\over 2\pi}\widehat{\psi_{\delta}^+}(0)  \int_{|\xi|\geq \xi_0 \sqrt n} e^{i\xi{t\over \sqrt n}}  e^{-{ \varrho^2\xi^2 \over 2}} \,\diff \xi. $$
	\medskip
	
	We will bound each $\Lambda_n^j$, $j=1,\ldots, 5$,  separately. We will use that $$\norm{\Phi_{n,\xi} +\Phi_{n}^{\star}  }_{\Cc^\alpha}\leq 2C \, n^{\alpha A} \leq 2C \, \sqrt[6] n$$ for every $\xi$, after Lemma \ref{lemma:norm-Phi-2} and the choice of $\alpha$ and $A$. 
	
	In order to bound $\Lambda_n^2$, we have, using the analyticity of $\xi\mapsto \oN_{i\xi}$, that
	$$ \Big\| \big(\oN_{{i\xi\over \sqrt n}}-\oN_0\big)(\Phi_{n,\xi} +\Phi_{n}^{\star}  ) \Big\|_\infty \lesssim  {|\xi|\over \sqrt n} \norm{\Phi_{n,\xi} +\Phi_{n}^{\star}  }_{\Cc^\alpha}\leq  { 2C |\xi|\over \sqrt[3] n}.$$
	Recall, from Lemma \ref{lemma:lambda-estimates}, that $\big|\lambda_{{i\xi\over \sqrt n}}^n\big|\leq e^{-{\varrho^2\xi^2\over 3}}$ for $|\xi|\leq \xi_0 \sqrt n$. 	Since $\|\widehat{\psi^\pm_{\delta}}\|_{\Cc^1}$ is bounded uniformly in $\delta$, we get 
	$$\sup_{t\in\R} \big|\Lambda_n^2(t)\big|\lesssim  \int_{-\infty}^{\infty} e^{-{\varrho^2\xi^2\over 3}}  {2C |\xi|\over \sqrt[3] n} \,\diff \xi \lesssim {1 \over \sqrt[3] n}. $$
	
	For  $\Lambda_n^3$, we use that $\norm{\oQ^n_z}_{\Cc^\alpha} \leq c \beta^n$ for $|z| \leq \xi_0$, where $c>0$ and $0<\beta<1$ are constants, see Proposition \ref{prop:spectral-decomp}. Therefore, for $|\xi|\leq \xi_0\sqrt n$,
	$$\Big\| \oQ_{{i\xi\over \sqrt n}}^n(\Phi_{n,\xi} +\Phi_{n}^{\star}  ) \Big\|_\infty \lesssim \beta^n\norm{\Phi_{n,\xi} +\Phi_{n}^{\star}  }_{\Cc^\alpha} \leq 2 C  \beta^n  \sqrt[6] n,    $$
which gives
	$$\sup_{t\in\R} \big|\Lambda_n^3(t)\big|\lesssim \int_{-\xi_0 \sqrt n}^{\xi_0 \sqrt n}  2 C  \beta^n  \sqrt[6] n \,\diff \xi = 4\xi_0 C \sqrt n  \beta^n \sqrt[6] n \lesssim \frac{1}{\sqrt n}.$$

	In order to bound  $\Lambda_n^4$,  we use that, after Proposition \ref{prop:spec-Pxi}, there are constants $A_\delta>0$ and $0<\rho_\delta<1$ such that $\norm{\oP^n_{i\xi}}_{\Cc^\alpha}\leq A_\delta \rho_\delta^n$ for all $\xi_0\leq |\xi|\leq \delta^{-2}$ and $n \geq 1$. Therefore,
	$$\sup_{t\in\R}  \big|\Lambda_n^4(t)\big|\lesssim \int_{\xi_0\sqrt n \leq|\xi|\leq\delta^{-2} \sqrt n} A_\delta \rho_\delta^n \sqrt[6] n  \,\diff \xi \leq  2  \delta^{-2}\sqrt n A_\delta \rho_\delta^n \sqrt[6] n \lesssim \frac{C_{\delta}'}{\sqrt n},$$
	for some constant $C_{\delta}'>0$.
	
	The modulus of the term $\Lambda_n^5$ is clearly $\lesssim 1/\sqrt n$, so it only remains to estimate $\Lambda_n^1$. For every $t \in \R$, we have $$\big| \Lambda_n^1(t) \big|\leq \Gamma_n^1+\Gamma_n^2+\Gamma_n^3,$$ where 
	$$\Gamma_n^1:= {1\over 2\pi}\int_{-\xi_0 \sqrt n}^{\xi_0 \sqrt n} \Big| \widehat{\psi_{\delta}^+}\Big({\xi\over \sqrt n}\Big) \Big| \, \big|\lambda_{{i\xi\over \sqrt n}}^n \big| \cdot\Big| \oN_0(\Phi_{n,\xi} +\Phi_{n}^{\star}  )- 1 \Big|  \,\diff \xi , $$
	$$\Gamma_n^2:= {1\over 2\pi}\int_{-\xi_0 \sqrt n}^{\xi_0 \sqrt n} \big|\lambda_{{i\xi\over \sqrt n}}^n\big|\cdot \Big| \widehat{\psi_{\delta}^+}\Big({\xi\over \sqrt n}\Big) -\widehat{\psi_{\delta}^+}(0) \Big|  \,\diff \xi $$
	and
	$$\Gamma_n^3:= {1\over 2\pi}\int_{-\xi_0 \sqrt n}^{\xi_0 \sqrt n} \big| \widehat{\psi_{\delta}^+}(0) \big| \cdot\Big| e^{-i \xi \sqrt n\gamma} \lambda_{{i\xi\over \sqrt n}}^n-    e^{-{ \varrho^2\xi^2 \over 2}}\Big|  \,\diff \xi. $$
	
	Recall that $\chi_k$ is bounded by one, supported by  $\Tc^u_k \subset \{|u| \geq |k| - 1\}$ (see Lemma \ref{lemma:partition-of-unity}) and $\frac12 \leq \varphi \leq 2$ by assumption. Therefore,
	$$\oN_0 (\chi_k\varphi) = \int_{\P^{d-1}} \chi_k\varphi \, \diff \nu \leq 2\, \nu \{|u| \geq |k| - 1\}   \lesssim e^{-\eta_* |k|},$$
	where in the last step we have used Property (1) of Definition \ref{def:admissible}.

	Using \eqref{eq:psi_xi+psi_T-2}, the assumption that $\oN_0 \varphi = 1$ and recalling that $t_k \in [-k-1,-k+1]$,  we get 
	\begin{align*}
	\Big| \oN_0(\Phi_{n,\xi} +\Phi_{n}^{\star}  )- 1 \Big| &= \Big| \oN_0(\Phi_{n,\xi} +\Phi_{n}^{\star}  )-\oN_0 \varphi \Big|\leq \sum_{|k|\leq  A \log n} \big| e^{-i \xi{t_k\over \sqrt n}}-1 \big|\oN_0 (\chi_k \varphi) \\
	&\lesssim \sum_{ k \geq 0} |\xi|{|t_k| \over \sqrt n}e^{- \eta_* |k|} \leq \sum_{ k \geq 0} |\xi|{|k| + 1 \over \sqrt n}e^{- \eta_* |k|} \lesssim {|\xi|\over \sqrt n}.
	\end{align*}
	
	Using that $\big|\lambda_{{i\xi\over \sqrt n}}^n\big|\leq e^{-{\varrho^2\xi^2\over 3}}$ for $|\xi|\leq \xi_0 \sqrt n$ (Lemma \ref{lemma:lambda-estimates}) and that  $\|\widehat{\psi^+_{\delta}}\|_{\Cc^1}$  is uniformly bounded,  we get that 
	$$ \Gamma_n^1\lesssim  \int_{-\xi_0 \sqrt n}^{\xi_0 \sqrt n}  \|\widehat{\psi^+_{\delta}}\|_{\infty} e^{-{\varrho^2\xi^2\over 3}} {|\xi|\over \sqrt n}  \,\diff \xi\lesssim {1 \over \sqrt n}$$
	and
	$$\Gamma_n^2\lesssim  \int_{-\xi_0 \sqrt n}^{\xi_0 \sqrt n} e^{-{\varrho^2\xi^2\over 3}} {|\xi|\over \sqrt n} \|\widehat{\psi^+_{\delta}}\|_{\Cc^1} \,\diff \xi\lesssim {1\over \sqrt n}.$$

	The bound $\Gamma_n^3\lesssim 1/\sqrt n$ follows by splitting the integral along the intervals $|\xi|\leq \sqrt[6] n$ and $\sqrt[6] n< |\xi| \leq \xi_0\sqrt n$ and using Lemma \ref{lemma:lambda-estimates}.
	
	 We conclude that $$\sup_{t\in\R} \big|\Lambda_n^1(t)\big|\lesssim \frac{1}{\sqrt n}.$$

	Gathering the above estimates, we obtain $$\sup_{t\in\R}\big|\Bc_n^+(t)-\Gc_n^+(t) \big|\lesssim \frac{1}{\sqrt n} + {1 \over \sqrt[3] n} +  \frac{1}{\sqrt n} +  \frac{C_{\delta}'}{\sqrt n} + \frac{1}{\sqrt n}.$$
Hence,  the above quantity is bounded by  $C_{\delta} / \sqrt[3] n$ for some constant $C_{\delta} > 0$. This finishes the proof of the lemma.
		\end{proof}
		
		We can now finish the proof Proposition \ref{prop:LLT-product-test-fcn}.
		
\begin{proof}[Proof Proposition \ref{prop:LLT-product-test-fcn}]
Recall from the beginning of the proof that we can assume that  ${1\over 2}\leq \varphi\leq 2$,  $\oN_0 \varphi=1$ and $\norm{\varphi}_{\Cc^\alpha}\leq 2$. 
From \eqref{eq:LLT-main-limit}, our goal is to show that $\sup_{t \in \R} |\Ac_n(t) - \Gc_n(t)|$ tends to zero as $n \to \infty$, where $\Ac_n(t)= \Ac_n(\psi \cdot \varphi)(t)$ and $$\Gc_n(t):= \Gc_n(\psi \cdot \varphi)(t) =  \frac{1}{\sqrt{2 \pi} \varrho} e^{-\frac{t^2}{2 \varrho^2 n}}  \int_{\R} \psi(s) \,\diff s \int_{\P^{d-1}} \varphi(w) \diff \nu(w) =  \frac{1}{\sqrt{2 \pi} \varrho} e^{-\frac{t^2}{2 \varrho^2 n}}  \int_{\R} \psi(s) \,\diff s,$$
since $\oN_0 \varphi = 1$.

Fix $0<\delta \leq 1$. Lemmas \ref{lemma:llt-ineq-1} and \ref{lemma-R-S} and the fact that $\Bc_n(t) \leq \Bc_n^+(t)$ give that
	$$\Ac_n(t) \leq \Gc_n^+(t) + {C_{\delta}\over \sqrt[3] n}+{C_1 \over \sqrt n} \quad \text{for all } \,\, t \in \R.$$

	Recall, from \eqref{eq:S-def},  that $\Gc_n^+(t) = {1\over \sqrt{2\pi} \, \varrho} e^{-{t^2\over2\varrho^2 n}}\int_{\R} \psi_{\delta}^+ (s)\,\diff s$. Hence, for every fixed $n$,  
	$$ \big|  \Gc_n^+(t) - \Gc_n(t)  \big| \leq {1\over \sqrt{2\pi} \, \varrho}  \big\|   \psi_{\delta}^+ -\psi\big\|_{L^1},$$ yielding
	$$ \Ac_n(t) -   \Gc_n(t) \leq   {1\over \sqrt{2\pi} \, \varrho}  \big\|   \psi_{\delta}^+ -\psi\big\|_{L^1} +   {C_{\delta}\over \sqrt[3] n}+{C_1 \over \sqrt n}.$$
 Therefore,	
\begin{equation} \label{eq:LLT-product-lim-sup}
\limsup_{n\to \infty}  \sup_{t\in\R} \big( \Ac_n(t) -  \Gc_n(t) \big) \leq  {1\over \sqrt{2\pi} \, \varrho}  \big\|   \psi_{\delta}^+ -\psi\big\|_{L^1}.
\end{equation}
		
		In order two deal with the lower bound, we repeat the above arguments replacing $t_k$ by a point $t_k' \in [-k-1,-k+1]$  such that $$\inf_{s \in [-k-1,-k+1]}  \psi(t+ \sigma(S_n,x) + s - n \gamma) = \psi(t+ \sigma(S_n,x) + t_k' - n \gamma)$$ and the upper approximation $\psi_{\delta}^+$ by the lower approximation $\psi^-_{\delta} \leq \psi$ from Lemma \ref{lemma:conv-fourier-approx}. See \cite{DKW:BE-LLT-coeff} for comparison. In this case, we conclude that		
\begin{equation} \label{eq:LLT-product-lim-inf}
\liminf_{n\to \infty}  \inf_{t\in\R} \big( \Ac_n(t) -   \Gc_n(t)  \big)  \geq  -   {1\over \sqrt{2\pi} \, \varrho}  \big\|    \psi^-_{\delta} - \psi \big \|_{L^1}.
\end{equation}

From \eqref{eq:LLT-product-lim-sup}, \eqref{eq:LLT-product-lim-inf} and the fact that $\big\|\psi_{\delta}^\pm -\psi\big\|_{L^1}$ tends to zero as $\delta \to 0$ by Lemma \ref{lemma:conv-fourier-approx}, we conclude that $\sup_{t \in \R} |\Ac_n(t) - \Gc_n(t)|$ tends to zero as $n \to \infty$. The proposition follows.
\end{proof}

We now give the proof of Theorem \ref{thm:LLT-general} in the general case.

\begin{proof}[Proof of  Theorem \ref{thm:LLT-general}]
Let $\Phi(s,w)$ be an arbitrary function on $\R \times \P^{d-1}$ which is continuous and has compact support. We need to show that \eqref{eq:LLT-main-limit} holds for $\Phi$. Fix $\varepsilon > 0$. Using Stone-Weierstrass Theorem, one can find an integer $N_\varepsilon >0$ and functions $\Phi^\pm$ of the form  $\Phi^\pm (s,w): =\sum_{1\leq j\leq N_\varepsilon} \psi_j^\pm(s) \varphi_j^\pm(w)$ where $\psi_j$ is continuous with compact support in $\R$ and $\varphi_j^\pm \in \Cc^\alpha(\P^{d-1})$, such that $\Phi^- \leq \Phi \leq \Phi^+$ and $\|\Phi^\pm - \Phi\|_{L^1(\Leb \otimes \nu)} \leq \sqrt{2 \pi} \varrho \, \varepsilon$. By Proposition \ref{prop:LLT-product-test-fcn}, \eqref{eq:LLT-main-limit} holds for $\Phi^\pm$. In particular, there exists $n_0 \in \N$ such that $$ \sup_{t\in\R} \big|\Ac_n(\Phi^\pm)(t) - \Gc_n(\Phi^\pm)(t)  \big| < \varepsilon \quad \text{for} \quad n \geq n_0.$$
Observe that, from the above properties of $\Phi^\pm$ and the definition of $\Gc_n$, we have that $\big| \Gc_n(\Phi)(t) - \Gc_n(\Phi^\pm)(t)  \big| \leq \varepsilon$ for all $t \in \R$ and every $n \geq 1$. Since $\Ac_n$ and $\Gc_n$ are positive functionals, we obtain the following inequalities for $n \geq n_0$ : 
$$\Ac_n(\Phi)(t) - \Gc_n(\Phi)(t) \leq \Ac_n(\Phi^+)(t) - \Gc_n(\Phi)(t) \leq \Ac_n(\Phi^+)(t) - \Gc_n(\Phi^+)(t)+ \varepsilon < 2 \varepsilon $$ and 
$$- 2\varepsilon <  \Ac_n(\Phi^-)(t) - \Gc_n(\Phi^-)(t) - \varepsilon \leq \Ac_n(\Phi^-)(t) - \Gc_n(\Phi)(t)\leq \Ac_n(\Phi)(t) - \Gc_n(\Phi)(t).$$ We conclude that $\big|\Ac_n(\Phi)(t) - \Gc_n(\Phi)(t) \big| < 2\varepsilon$ for all $t \in \R$ and $n \geq n_0$. Since $\varepsilon>0$ is arbitrary, this gives \eqref{eq:LLT-main-limit}. This finishes the proof of the theorem.
\end{proof}

\section{Remarks on the Local Limit Theorem} \label{sec:LLT-D}

In this section, we discuss a refinement of Theorem \ref{thm:LLT-general} that allows moderate deviations. More precisely,  we let the parameter $t$ in the LLT depend on $n$ in the range $|t| = o(\sqrt n)$, see \cite{petrov} for the case of sums of i.i.d.'s.  Notice that we use here a smaller class of admissible functions $\Lc_{s_0}(\eta_*, \alpha_*,A_*)$, which is introduced in Definition \ref{def:admissible-2} below.  We can recover Theorem \ref{thm:LLT-general} for those functions. As before, this class contains the functions $u=0$ and $u(x) = \log d(x,H_y)$, see Lemma \ref{lemma:log-dist-admissible-2}.

The Cram\' er function is an analytic function in the real line whose coefficients are polynomials in the derivatives of the function $\Lambda(s):= \log \lambda_s$ at $s=0$, where $\lambda_s$ is the leading eigenvalue of the perturbed Markov operator \eqref{eq:markov-op-def} for $z=s \in \R$, see Proposition \ref{prop:spectral-decomp}. More precisely, if $\gamma_m := \Lambda^{(m)}(0)$, $m \geq 1$, then
\begin{equation*}
\zeta(t)=\frac{\gamma_3}{6\gamma_2^{3/2} } + \frac{\gamma_4\gamma_2-3\gamma_3^2}{24\gamma_2^3}t
+ \frac{\gamma_5\gamma_2^2-10\gamma_4\gamma_3\gamma_2 + 15\gamma_3^3}{120\gamma_2^{9/2}}t^2 + \cdots, 
\end{equation*}
which is convergent near $t=0$. See \cite{petrov} for more on this and similar functions and their role in the  classical limit theorems with moderate deviations for sums of independent random variables.

\begin{theorem}[Local Limit Theorem with moderate deviations] \label{thm:LLT-D}
	Let $\mu$ be a  probability measure on $\GL_d(\R)$.  Assume that $\mu$ has a finite exponential moment and  that $\Gamma_\mu$ is proximal and strongly irreducible.  Let $\gamma, \varrho$ and $\nu$ be as in the Introduction. If $u\in \Lc_{s_0}(\eta_*, \alpha_*,A_*)$ and $\Phi$  is a continuous function with compact support on $\R\times \P^{d-1}$, then, for $|t| = o(\sqrt n)$ we have
\begin{align*}
\sqrt{n}  \,\E \Big( \Phi\Big( \sigma(S_n,x)+u(S_nx)& - n \gamma - \sqrt n \varrho t, S_n x\Big) \Big) \\ &= \frac{1}{\sqrt{2\pi} \varrho}   e^{-\frac{t^2}{2} + \frac{t^3}{\sqrt n} \zeta  (\frac{t}{\sqrt n})} \Big( \int_{\R\times\P^{d-1}} \Phi(s,w)\,\diff s\diff \nu(w) + o(1) \Big)
\end{align*}
as $n \to \infty$.
\end{theorem}

When $u=0$ or $u(x) = \log d(x,H_y)$ and $\Phi$ is of the form $\Phi(s,w) = \psi(s) \varphi(w)$ with $\psi$ continuous with compact support in $\R$ and $\varphi \in \Cc^\alpha(\P^{d-1})$, the above result is due to Xiao-Grama-Liu, see \cite{xiao-grama-liu,xiao-grama-liu:coeff}. As for many of the known limit theorems for products of random matrices, the case of the norm cocycle ($u=0$) is treated via spectral theory by adapting standard techniques from the case of sums of independent random variables \cite{petrov}, see \cite{xiao-grama-liu}. The case $u(x) = \log d(x,H_y)$, which by \eqref{eq:coeff-split} encode the coefficients of $S_n$, can be handled by a well-chosen cut-off procedure, using the approach introduced in \cite{DKW:BE-LLT-coeff}, see  \cite[Section 7]{xiao-grama-liu:coeff}. This approach also works here thanks to Lemma \ref{lemma:partition-of-unity}. For this reason, we omit the full proof of  Theorem \ref{thm:LLT-D} and only highlight its main steps. In order to do that, some notation needs to be introduced. We refer to \cite{xiao-grama-liu} for more details.

Consider the perturbed Markov operators $\oP_z$ defined in \eqref{eq:markov-op-def}. By Proposition \ref{prop:spectral-decomp} we have a decomposition  $\oP_z = \lambda_z \oN_z + \oQ_z$ for $z$ small, where $\lambda_z$ is the leading eigenvalue of $\oP_z$ and $\oN_z$ is a rank-one projection. Let $r_z \in \Cc^\alpha(\P^{d-1})$ be the eigenfunction of $\oP_z$ associated with $\lambda_z$ normalized so that $\langle \nu, r_z \rangle = 1$.

Let $s_0 > 0$ be a small constant. For $s \in (-s_0,s_0)$, it can be shown that $\lambda_s >0$ and $r_s$ is a strictly positive function, see Proposition \ref{prop:spectral-decomp}. In this case, for $n \geq 1$, the functions
$$q_n^s(x,g):= \frac{e^{s\sigma(g,x)}}{\lambda_s^n} \frac{r_s(gx)}{r_s(x)} \quad x \in \P^{d-1}, g \in G$$
are such that $\Q^x_{s,n}:= q_n^s(x,S_n) \mu^{\otimes n}$, $n \geq 1$, form a projective system of probability measures  on $G^{\N^*}$. By Kolmogorov extension theorem, there exists a unique probability measure $\Q^x_s$
on $G^{\N^*}$ having $\Q^x_{s,n}$ as marginals. It follows from this construction that
\begin{align} 
 \E_{\Q^x_s} \big(\Psi(g_1,\sigma(g_1,x),& \ldots, S_n,\sigma(S_n,x)) \big) \nonumber \\ &= \frac{1}{\lambda_s^n r_s(x)} \E\Big( r_s(S_n x) e^{s \sigma(S_n,x)} \Psi\big(g_1,\sigma(g_1,x), \ldots, S_n,\sigma(S_n,x) \big) \Big) \label{eq:change-of-measure}
\end{align}
for any bounded measurable function $\Psi$ on $(\R\times\P^{d-1})^n$. Observe that $\Q^x_0 = \mu^{\otimes \N^*}$, where $\mu$ is the original probability measure on $G$ defining the random walk $S_n$.

The measure $\Q^x_s$ defines a new Markov chain on $\P^{d-1}$ whose transition operator is $\oU_s \varphi = \lambda_s^{-1} r_s^{-1} \oP_s(\varphi r_s)$. When $s=0$, we recover the original Markov process induced by $\mu$. The spectral theory for the operator $\oU_s$ and the perturbations
\begin{equation} \label{eq:perturbed-markov-s}
\oR_{s,i\xi} \varphi(x) := \E_{\Q^x_s} \Big( e^{i\xi (\sigma(g,x)- \Lambda'(s))} \varphi(g x) \Big),
\end{equation}
 where $\Lambda(s) = \log \lambda_s$ is defined above, is analogous to that of $\oP$ and its perturbations $\oP_{i\xi}$ discussed in Section \ref{sec:prelim}. In particular, $\oU_s$ admits a unique invariant measure, denoted by $\pi_s$, and it has a spectral gap in a suitable H\"older space, see \cite{xiao-grama-liu}.

\begin{definition} \rm \label{def:admissible-2} 
Let $s_0 \in\R^+$, $\eta_*>0$, $0< \alpha_* \leq 1$ and $A_*>0$ be constants. A function  $u:\P^{d-1}\to\R\cup \{\pm \infty\}$ belongs to the class $\Lc_{s_0}(\eta_*, \alpha_*,A_*)$ if $u$ is continuous, satisfies Property (2) from Definition  \ref{def:admissible} and
\begin{enumerate}
\item[(1')] If $s \in (-s_0,s_0)$, then $\pi_s \{|u|\geq t\} \leq A_*e^{-\eta_* t}$ for every $t\geq 0$.
\end{enumerate}
\end{definition}

Observe that $\pi_0 = \nu$, so condition (1') above contains condition (1) from Definition \ref{def:admissible}. The following result is contained in \cite{grama-quint-xiao}.

\begin{lemma}  \label{lemma:log-dist-admissible-2}
 There are constants $s_0 \in\R$, $\eta_*>0$  and $A_*>0$ such that for every $y\in (\P^{(d-1)})^*$ the function $\log d(x,H_y)$ belongs to $\Lc_{s_0}(\eta_*, 1, A_*)$.
\end{lemma}

 Using the properties of the operators $\oU_s$ described above, we can prove the following analogue of Proposition \ref{prop:LDT-admissible}. When $u(x) = \log d(x,H_y)$, a similar property was obtained in \cite{grama-quint-xiao}.

\begin{proposition}   \label{prop:LDT-admissible-2}
Let $u \in \Lc_{s_0}(\eta_*, \alpha_*,A_*)$. There exist constants $c>0$ and $n_0, N \in \N$  such that, for $x\in \P^{d-1}$ and $s \in (-s_0,s_0)$, one has 
	$$   \Q^x_s \big\{|u(S_\ell x)| \geq  m  \big\}  \leq e^{-c m} \quad \text{for} \quad \ell \geq N m \geq n_0, \,\, \ell, m \in \N.$$
\end{proposition}

We now sketch the proof of Theorem \ref{thm:LLT-D}. As in the proof of Theorem \ref{thm:LLT-general} given in Section  \ref{sec:LLT}, by standard approximation arguments, it is enough to prove Theorem \ref{thm:LLT-D} when $\Phi = \psi \varphi$, where $\psi$ is a continuous function with compact support on $\R$ and $\varphi \in \Cc^\alpha(\P^{d-1})$. In this case,  we need to show that
\begin{equation} \label{eq:LLT-D-main-limit}
 \Ic_n (t) =  \frac{1}{\sqrt{2 \pi} \varrho}  e^{-\frac{t^2}{2} + \frac{t^3}{\sqrt n} \zeta  (\frac{t}{\sqrt n})}\Big( \int_{\P^{d-1}} \varphi \, \diff \nu \, \int_{\R} \psi(v) \,\diff v + o(1) \Big)
\end{equation}
as $n \to \infty$, where
\begin{equation*} 
\Ic_n (t) := \sqrt{n}  \,\mathbf E \Big( \psi\big(\sigma(S_n,x)+u(S_nx) - n \gamma - \sqrt n \varrho t \big) \, \varphi(S_n x) \Big)
\end{equation*}

Recall that $\Lambda(s) = \log \lambda_s$ and $\lambda_s = 1 + \gamma s + O(s^2)$, see Proposition \ref{prop:spectral-decomp} . It follows that $\Lambda'(0) = \gamma$. For each $t$ such that $|t| = o(\sqrt n)$ choose $s \approx \frac{t}{\varrho \sqrt n}$ such that $\Lambda'(s) - \Lambda'(0) = \frac{\varrho t}{\sqrt n}$. Then, it can be shown that $$s \Lambda'(s) - \Lambda(s) = \frac{t^2}{2n} - \frac{t^3}{n^{3/2}} \zeta \Big( \frac{t}{\sqrt n}\Big),$$
where $\zeta$ is the Cram\' er function introduced above. Together with the change of measure formula \eqref{eq:change-of-measure}, $\Ic_n (t)$ is equal to
  \begin{align*}
  r_s(x) e^{-n(s \Lambda'(s) - \Lambda(s))} \sqrt n  \E_{\Q^x_s} \Big(  e^{-s(\sigma(S_n,x) - n \Lambda'(s))} \psi \big(\sigma(S_n,x)  + u(S_n x) - n \Lambda'(s)\big) (\varphi r_s^{-1})(S_n x) \Big).
  \end{align*}
 Notice that   $e^{-n(s \Lambda'(s) - \Lambda(s))} =  e^{-\frac{t^2}{2} + \frac{t^3}{\sqrt n} \zeta  (\frac{t}{\sqrt n})}$. Also, $s \to 0$ when $n \to \infty$, so $r_s(x) \to r_0(x) = 1$ uniformly. Therefore, if we set
 \begin{equation*}
 \Jc_n(f) (s) := \sqrt n   \, \E_{\Q^x_s} \Big(e^{-s(\sigma(S_n,x) - n \Lambda'(s))} f \big(\sigma(S_n,x)  + u(S_n x) - n \Lambda'(s), S_n x\Big)
  \end{equation*}
for a function $f$ on $\R \times \P^{d-1}$, we see that proving \eqref{eq:LLT-D-main-limit} is equivalent to showing that
  \begin{equation} \label{eq:LLT-D-main-limit-2}
  \lim_{n\to \infty} \Big| \Jc_n (\psi \cdot \varphi r_s^{-1}) (s) -  \frac{1}{\sqrt{2 \pi} \varrho} \int_{\P^{d-1}} \varphi \, \diff \nu \, \int_{\R} \psi(v) \,\diff v\Big| =0.
  \end{equation}

In order to prove \eqref{eq:LLT-D-main-limit-2},  similar arguments as the ones from Section \ref{sec:LLT} can be used.  Using the partition $\chi_k$, $k \in \Z$ from Lemma \ref{lemma:partition-of-unity}, we can write $$ \Jc_n (\psi \cdot \varphi r_s^{-1}) (s)  = \sum_{|k| \leq A \log n} \Jc_n(\psi \cdot \chi_k \varphi r_s^{-1}) (s) + \Jc_n(\psi \cdot \Phi_{n}^{\star}  r_s^{-1}) (s):= \Jc^1_n (s) + \Jc^2_n (s),$$
where   $\Phi_{n}^{\star}  (w):= \varphi(w) - \sum_{|k| \leq A \log n} (\chi_k\varphi)(w)$ as in the preceding sections and $A>0$ is a large constant. By choosing $A>0$ large enough and using Proposition \ref{prop:LDT-admissible-2}, one can show that $|\Jc^2_n (s)| = o(1)$. Hence, it remains to show that 
  \begin{equation}   \label{eq:LLT-D-main-limit-3}
  \lim_{n\to \infty} \Big| \Jc^1_n (s) -  \frac{1}{\sqrt{2 \pi} \varrho} \int_{\P^{d-1}} \varphi \, \diff \nu \, \int_{\R} \psi(v) \,\diff v\Big| =0.
  \end{equation}
  
  By approximating $\psi$ above and below by the functions $\psi_\delta^\pm$ given by Lemma \ref{lemma:conv-fourier-approx}, considering suitable translates $\psi_{s,k}^\pm$ of $\psi_\delta^\pm$ and using the Fourier inversion formula as in Section \ref{sec:LLT}, the quantity $\Jc_n(\psi \cdot \chi_k \varphi r_s^{-1}) (s)$ can be estimated in terms of $\widehat{\psi_\delta^\pm}$ and the operator $\oR_{s,i\xi}$ defined in \eqref{eq:perturbed-markov-s} acting on $\chi_k \varphi r_s^{-1}$. Compare with \eqref{eq:E_n-psi+}, which corresponds to the case $s=0$. Then, the spectral properties of $\oR_{s,i\xi}^n$ from  \cite{xiao-grama-liu}, which are analogous to the ones of Proposition \ref{prop:spectral-decomp}, can be used in the same spirit of the estimates from Section \ref{sec:LLT} to prove \eqref{eq:LLT-D-main-limit-3}, compare with \cite[Section 7]{xiao-grama-liu:coeff}. Theorem \ref{thm:LLT-D} follows.

\end{document}